\documentclass[12pt,draftcls,onecolumn]{IEEEtran}
\usepackage{amsmath,epsfig,cite,amsfonts,amssymb,psfrag,subfigure,balance}
\usepackage{rotating}   
\usepackage{citesort}                                                          

\IEEEoverridecommandlockouts                              
\def\ex{\mathbb{E}}

\def\d{\mathcal{D}}
\def\cc{\mathcal{C}}
\def\e{\mathcal{E}}
\def\r{\mathcal{R}}
\def\ps{P_S}
\def\pra{P_R}
\def\pr{P_r}
\def\a{\alpha}
\def\b{\beta}

\def\lm{\lambda}
\def\l1{L_1}
\def\l2{L_2}
\def\nr{N_r}
\def\nd{N_d}
\def\n{N}
\def\nt{\tilde{N}}
\def\sm{\text{-}}
\def\dt{\text{det}}
\def\lmm{\Lambda}
\def\l1{L_1}
\def\l2{L_2}

\def\nt{\tilde{N}}
\def\sm{\text{-}}

\newtheorem{theorem}{\bf{Theorem}}[section]

\newtheorem{lemma}{\bf{Lemma}}[section]

\newtheorem{proposition}{\bf{Proposition}}[section]
\newtheorem{definition}{\bf{Definition}}[section]
\newtheorem{remark}{\bf{Remark}}[section]
\newtheorem{proof}{\bf{Proof}}[section]
\begin{document}

\title{Stabilization of Linear Systems Over Gaussian Networks}
\author{\authorblockN{Ali A.~Zaidi\authorrefmark{1}, Tobias J. Oechtering\authorrefmark{1}, Serdar Y\"{u}ksel\authorrefmark{2} and Mikael Skoglund\authorrefmark{1}}\\
\authorblockA{\authorrefmark{1} KTH Royal Institute of Technology, Stockholm, Sweden\\
\authorblockA{\authorrefmark{2}Department of Mathematics and Statistics,
Queen's University, Canada}}
\thanks{Parts of this work were presented in IEEE ICCA June 2010, Reglerm\"{o}te June 2010, and ACC, June 2011.}}
\maketitle
\vspace{-1cm}
\begin{abstract}
  The problem of remotely stabilizing a noisy linear time invariant
  plant over a Gaussian relay network is addressed. The network is
  comprised of a sensor node, a group of relay nodes and a remote
  controller. The sensor and the relay nodes operate subject to an
  average transmit power constraint and they can cooperate to
  communicate the observations of the plant's state to the remote
  controller. The communication links between all nodes are modeled as
  Gaussian channels. Necessary as well as sufficient conditions for
  mean-square stabilization over various network topologies are
  derived. The sufficient conditions are in general obtained using
  delay-free linear policies and the necessary conditions are obtained
  using information theoretic tools. Different settings where linear policies are optimal, asymptotically optimal (in certain parameters of the system) and suboptimal have been identified. For the case with noisy multi-dimensional sources controlled over scalar channels, it is shown that linear time varying policies lead to minimum capacity requirements, meeting the fundamental lower bound. For the case with noiseless sources and parallel channels, non-linear policies which meet the lower bound have been identified.
\end{abstract}

\vspace{-.4cm}
\section{Introduction}
The emerging area of networked control systems has gained significant attention in recent years due to its potential applications in many areas such as machine-to-machine communication for security, surveillance, production, building management, and traffic control. The idea of controlling dynamical systems over communication networks is supported by the rapid advance of wireless technology and the development of cost-effective and energy efficient devices (sensors), capable of sensing, computing, and transmitting. This paper considers a setup in which a sensor node communicates the observations of a linear dynamical system (plant) over a network of wireless nodes to a remote controller in order to stabilize the system in closed-loop. The wireless nodes have transmit and receive capability and we call them \emph{relays}, as they relay the plant's state information to the remote controller. We assume a transmit power constraint on the sensor and relays, and the wireless links between all agents (sensor, relays, and controller) are modeled as Gaussian channels. The objective is to study stabilizability of the plant over Gaussian networks.


\subsection{Problem Formulation}
Consider a discrete linear time invariant system, whose state equation is given by
\begin{eqnarray}
   X_{t+1}&=&A X_t + B U_t + W_t, \label{eq:stateEquation}
\end{eqnarray}
where $X_t\in\mathbb{R}^n,U_t\in\mathbb{R}^m$, and
$W_t\in\mathbb{R}^n$ are state, control, and plant noise
The initial state $X_0$ is a random variable with bounded differential
entropy $|h(X_0)|<\infty$ and a given covariance matrix
$\lmm_0$. The plant noise $\{W_t\}$ is a zero mean white Gaussian
noise sequence with variance $K_W$ and it is assumed to be independent
of the initial state $X_0$. The matrices $A$ and $B$ are of appropriate dimensions and the
pair $(A,B)$ is controllable. Let $\{\lm_1,\lm_2,\dots,\lm_n\}$ denote
the eigenvalues of the system matrix $A$. Without loss of generality
we assume that all the eigenvalues of the system matrix are outside
the unit disc, i.e., $|\lm_i|\geq 1$. The unstable modes can
be decoupled from the stable modes by a similarity transformation. If
the system in \eqref{eq:stateEquation} is one-dimensional then $A$ is
scalar and we use the notation $A=\lm$. We consider a remote control
setup, where a sensor node observes the state process and transmits it
to a remotely situated controller over a network of relay\footnote{A
  relay is a communication device whose sole purpose is to support
  communication from the information source to the destination. In our
  setup the relay nodes cooperate to communicate the state process
  from sensor to the remote controller. If the system design objective
  is to replace wired connections, then relaying is a vital approach
  to communicate over longer distances.} nodes as shown in
Fig. \ref{fig:SystemDiagram}. The communication links between nodes
are modeled as white Gaussian channels, which is why we refer to it as
a Gaussian network. In order to communicate the observed state value
$X_t$, an encoder $\mathcal{E}$ is lumped with the observer
$\mathcal{O}$ and a decoder $\mathcal{D}$ is lumped with the
controller $\mathcal{C}$. In addition there are $L$ relay nodes
$\{\mathcal{R}_i\}^L_{i=1}$ within the channel to support
communication from $\mathcal{E}$ to $\mathcal{D}$. At any time instant
$t$, $S_{e,t}$ and $R_t$ are the input and the output of the network
and $U_t$ is the control action. Let $f_t$ denote the observer/encoder
policy such that $S_{e,t}=f_t(X_{[0,t]},U_{[0,t-1]})$, where
$X_{[0,t]}:= \{X_0, X_1,\dots,X_t\}$ and we have the following
average transmit power constraint $\lim_{T\rightarrow \infty}\frac{1}{T}\sum^{T-1}_{t=0}\mathbb{E}[S^2_{e,t}]\leq P_S$. Further let $\pi_t$ denote the decoder/controller policy, then
$U_t=\pi_t\left(R_{[0,t]}\right)$. The objective in this paper is to
find conditions on the system matrix $A$ so that the plant in
(\ref{eq:stateEquation}) can be mean square stabilized over a given
Gaussian network.

\begin{definition}
A system is said to be \emph{mean square stable} if there exists a constant $M<\infty$ such that $\ex[\|X_t\|^2]<M$ for all $t$.
\end{definition}

\begin{figure}[!t]
    \centering
    \psfrag{p}[][][3]{\begin{sideways}Plant\end{sideways}}
    \psfrag{oe}[][][3]{$\mathcal{O}/\mathcal{E}$}
    \psfrag{dc}[][][3]{$\mathcal{D}/\mathcal{C}$}
    \psfrag{rn}[][][2]{$\mathcal{R}_L$}
    \psfrag{r1}[][][2]{$\mathcal{R}_1$}
    \psfrag{r2}[][][2]{$\mathcal{R}_2$}
    \psfrag{r3}[][][2]{$\mathcal{R}_3$}
    \psfrag{ch}[][][2]{AWGN Relay Channel}
    \psfrag{nfch}[][][2]{\begin{sideways}Noiseless Feedback Communication Channel\end{sideways}}
    \psfrag{x}[][][2.5]{$X_t$}
    \psfrag{st}[][][2.5]{$S_{e,t}$}
    \psfrag{rt}[][][2.5]{$R_t$}
    \psfrag{ut}[][][2.5]{$U_t$}
    \resizebox{7 cm}{!}{\epsfbox{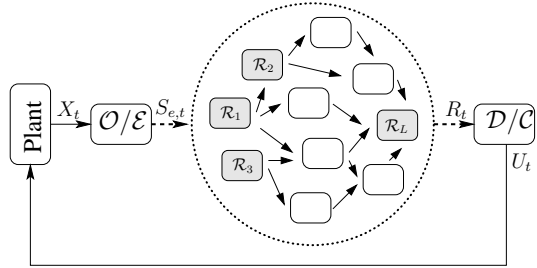}}
    \vspace{-.2cm}
  \caption{The unstable plant has to be controlled over a Gaussian relay network.}\label{fig:SystemDiagram}
\end{figure}


\vspace{-.3cm}
\subsection{Literature Review}
Important contributions to control over communication channels include \cite{BansalBasar89,Baillieul99,WongBrockett99,NairEvans00,EliaMitter01,NairEvans04,Elia04,MatveevSavkin04,TatikondaMitter04b,TatikondaMitter04c,MartinElia06,SahaiMitter06,MatveevSavkin07,BraslavskyFreudenberg07,MartinsDahleh08,MiddletonBraslavsky09,YukselTatikonda09,YukselBasar10,FarhadiAhmed11}.
The problem of remotely controlling dynamical systems over communication channels is studied with methods from stochastic control theory and information theory. The seminal paper by Bansal and Ba\c{s}ar \cite{BansalBasar89} used fundamental information theoretic arguments to obtain optimal policies for LQG control of a first order plant over a point to point Gaussian channel. Minimum rate requirements for stabilizability of a noiseless scalar plant were first established in \cite{Baillieul99,WongBrockett99} followed by \cite{NairEvans00}. Further rate theorems for stabilization of linear plants over some discrete and continuous alphabet channels can be found in \cite{TatikondaMitter04c,BraslavskyFreudenberg07,MartinsDahleh08,CharalambousFarhadi08,MiddletonBraslavsky09,YiqianChen09,YukselBasar10,FreudenbergSolo10,ShuMiddleton11,YukselIT12,SilvaQuevedo10,VargasSilva12}. The papers \cite{BansalBasar89,TatikondaMitter04b,TatikondaMitter04c,BraslavskyFreudenberg07,MiddletonBraslavsky09,YiqianChen09,FreudenbergSolo10,YukselTatikonda09,ShuMiddleton11,YukselBasar10,SilvaQuevedo10,VargasSilva12} addressing control over Gaussian channels are more relevant to our work. In \cite{BansalBasar89} linear sensing and control policies are shown to be optimal for the LQG control of a first order linear plant over a point-to-point Gaussian channel. A necessary condition for stabilization relating eigenvalues of the plant to the capacity of the Gaussian channel first appeared in \cite{TatikondaMitter04b,TatikondaMitter04c}. Some important contributions on stabilization over Gaussian channels with average transmit power constraints have been made in \cite{BraslavskyFreudenberg07,MiddletonBraslavsky09,FreudenbergSolo10,YiqianChen09,ShuMiddleton11,KumarLaneman11,VargasSilva12}. In \cite{BraslavskyFreudenberg07} sufficient conditions for stabilization of both continuous time and discrete time multi-dimensional plants over a scalar white Gaussian channel were obtained using linear time invariant (LTI) sensing and control schemes. It was shown in \cite{BraslavskyFreudenberg07,FreudenbergSolo10} that under some assumptions there is no loss in using LTI schemes for stabilization, that is the use of non-linear time varying schemes does not allow stabilization over channels with lower signal-to-noise ratio. The stability results were extended to a colored Gaussian channel in \cite{MiddletonBraslavsky09}. In \cite{YukselBasar10} the authors considered noisy communication links between both sensor--controller and controller--actuator and presented necessary and sufficient conditions for mean square stability. Stabilization of noiseless LTI plants over parallel white Gaussian channels subject to transmit power constraint has been studied in \cite{YiqianChen09,ShuMiddleton11,KumarLaneman11,VargasSilva12}. The paper \cite{YiqianChen09} considers output feedback stabilization and \cite{ShuMiddleton11} considers state feedback stabilization, and they both derive necessary and sufficient conditions for stability under a total transmit power constraint. The necessary condition derived in \cite{ShuMiddleton11} for mean-square stabilization of discrete time LTI plants over parallel Gaussian channels is not tight in general and its achievability is not guaranteed by LTI schemes. The paper \cite{VargasSilva12} focuses on  mean-square stabilization of two-input two-output system over two parallel Gaussian channels. By restricting the study to LTI schemes and assuming individual power constraint on each channel, the authors derive tight necessary and sufficient conditions for both state feedback and output feedback architectures. Realizing that LTI schemes are not optimal in general for stabilization over parallel channels \cite{ShuMiddleton11},
the paper \cite{KumarLaneman11} proposes a non-linear time invariant scheme for stabilization of a scalar noiseless plant over a parallel Gaussian channel using the idea that independent information should be transmitted on parallel channels \cite{VaishampayanThesis,YukselTatikonda09}. The problem of finding a tight necessary and sufficient condition for stabilization of an $m$-dimensional plant over an $n$-dimensional parallel Gaussian channel is still open, which we investigate in this paper.


As summarized above, the previous works on control over Gaussian channels have mostly focused on situations where there is no intermediate node between the sensor and the remote controller. The problems related to control over Gaussian networks with relay nodes are largely open. Such problems are hard because a relay network can have an arbitrary topology and every node within the network can have memory and can employ any transmit strategy. The papers \cite{Tatikonda03} and \cite{GuptaHassibi09} have derived conditions for stabilization over networks with digital noiseless channels and analog erasure channels respectively, however those results do not apply to noisy networks. In \cite{SahaiMitter06,YukselIT12} moment stability conditions in terms of error exponents have been established. However, even a single letter expression for channel capacity of the basic three-node Gaussian relay channel \cite{InfoBook} is not known in general. In \cite{GastparVetterli05} Gastpar and Vetterli determined capacity of a large Gaussian relay network in the limit as the number of relays tends to infinity. The problem of control over Gaussian relay channels was first introduced in \cite{zaidiICCA10,zaidiReglermote} and further studied in \cite{zaidiACC11,KumarLanemanCDC10}. The papers \cite{zaidiICCA10,zaidiReglermote,zaidiACC11,KumarLanemanCDC10} derived sufficient conditions for mean square stability of a scalar plant by employing linear schemes over Gaussian channels with single relay nodes. In this paper we consider more general setups with multiple relays and multi-dimensional plants. We also derive necessary conditions along with sufficient conditions and further discuss how good linear policies are for various network topologies. In particular this paper makes the following contributions:

\vspace{-.3cm}
\subsection{Main Contributions}
\begin{itemize}
\item In Sec.~\ref{sec:NecConditionGen} we obtain a necessary condition for
mean square stabilization of the linear system in
\eqref{eq:stateEquation} over the general relay network depicted in Fig.~\ref{fig:SystemDiagram}.
\item In Sec.~\ref{sec:Cascade}--\ref{sec:NonOrthogonal} we derive necessary as well as sufficient conditions for
  stabilization over some fundamental network topologies such as
cascade network, parallel network, and non-orthogonal network, which serve as building blocks for a large class of
Gaussian networks (see Figures \ref{fig:CascadeNetwork}, \ref{fig:ParallelNetwork}, \ref{fig:HalfDuplexNetwork}, pp. 7, 10, 13). Necessary conditions are obtained using information
  theoretic tools. Sufficient conditions are obtained using linear
  schemes.
\item Sub-optimality of linear policies is discussed and some insights on optimal schemes are presented. In some cases linear schemes can be
  asymptotically optimal and in some cases exactly optimal.
\item A linear time varying scheme is proposed in Sec.~\ref{sec:multiDimension}, which is optimal for stabilization of noisy multi-dimensional plants over the point-to-point scalar Gaussian channels.
\item The minimum rate required for stabilization of multi-dimensional plants over parallel Gaussian channels is established in Sec. \ref{sec:Parallel}, which is achievable by a non-linear time varying scheme for noiseless plants.
\end{itemize}


\vspace{-.3cm}
\section{Necessary Condition for Stabilization}\label{sec:NecConditionGen}
In the literature \cite{Elia04,MartinsDahleh08,SilvaOstergaard11,YukselIT12}, there exist a variety of information rate inequalities characterizing fundamental limits on the performance of linear systems controlled over communication channels. In the following we state a relationship which gives a necessary condition for mean square stabilization over the general network depicted in Fig.~\ref{fig:SystemDiagram}.
\begin{theorem}\label{thm:NecGeneral}
  If the linear system in \eqref{eq:stateEquation} is mean square
  stable over the Gaussian relay network, then
\begin{align}
  \log\left(|\dt\left(A\right)|\right)\leq\liminf_{T\rightarrow\infty}
  \frac{1}{T} I(\bar{X}_{[0,T-1]} \rightarrow R_{[0,T-1]}),
\end{align}
where $\{\bar{X}_t\}$ denotes the uncontrolled state process obtained by substituting $U_t=0$ in \eqref{eq:stateEquation}, i.e., $\bar{X}_{t+1}=A\bar{X}_t+W_t$, the notation $|\dt\left(A\right)|$ represents the absolute value of determinant
of matrix $A$ and $$I(\bar{X}_{[0,T-1]} \rightarrow R_{[0,T-1]})=\sum^{T-1}_{t=0} I
\left(\bar{X}_{[0,t]};R_t|R_{[0,t\sm 1]}\right)$$ is the directed
information from the uncontrolled state process
$\{\bar{X}_{[0,T-1]}\}$ to the sequence of variables $\{R_{[0,T-1]}\}$
received by the controller over the network of relay nodes.
\end{theorem}
\begin{proof}
The proof is given in Appendix \ref{apx:NecGeneral}, which essentially follows from the same steps as in the proof of Theorem 4.1 in \cite{YukselIT12}, however, with some differences due to the network structure. Similar constructions can also be found in \cite{MartinsDahleh08,SilvaOstergaard11}.
\end{proof}

\section{Cascade (Serial) Network}\label{sec:Cascade}
In this section we consider a cascade network of half-duplex relay nodes. A node which is capable of transmitting and receiving signals simultaneously using the same frequency band is known as \emph{full-duplex} while a \emph{half-duplex} node cannot simultaneously receive and transmit signals. In practice it is expensive and hard to a build a communication device which can transmit and receive signals at the same time using the same frequency, due to the self-interference created by the transmitted signal to the received signal. Therefore half-duplex systems are mostly used in practice. Consider a cascade network comprised of $L-1$ half-duplex relay nodes depicted in Fig.~\ref{fig:CascadeNetwork}, where the state encoder $\e$ observes the state of the system and transmits its signal to the relay node $\r_1$. The relay node $\r_1$ transmits a signal to the relay node $\r_2$ and so on. Finally the state information is received at the remote decoder/controller $\d$ from $\r_{L-1}$. The communication within the network takes place such that only one node is allowed to transmit at every time step. That is, if in a time slot $\r_i$ transmits signal to $\r_{i+1}$, then all the remaining nodes in the network are considered to be silent in that time slot. At any time step $t$, $S_{e,t}$ is the signal transmitted from $\e$ and $S^i_{r,t}$ is the signal transmitted from $\r_i$, which are given by
\begin{align}\label{eq:inOutCascadeRelay1}
&S_{e,t}=f_t\left(X_{[0,t]},U_{[0,t-1]}\right) \quad \forall t:t=1+nL, n\in\mathbb{N}, \qquad S_{e,t}=0 \quad \text{otherwise}, \nonumber \\
&S^i_{r,t}=g^i_t\left(Y^i_{[0,t]}\right) \quad \forall t:t=1+i+nL, n\in\mathbb{N}, \qquad S^i_{r,t}=0 \quad \text{otherwise},
\end{align}
where  $\mathbb{N}=\{0,1,2,\dots\}$, $f_t:\mathbb{R}^{2t\sm1}\rightarrow\mathbb{R}$, $g^i_t:\mathbb{R}^t\rightarrow\mathbb{R}$ such that $\ex\left[f^2_t\left(X_{[0,t]},U_{[0,t-1]}\right)\right]=L\ps$, $\ex\left[\left(g^i_t\left(Y_{[0,t]}\right)\right)^2\right]=L\pr^i$, $\sum^{L\sm1}_{i=1}\pr^i\leq\pra$. The signal received by $\r_i$ is
\begin{align}\label{eq:inOutCascadeRelay2}
Y^1_t=S_{e,t}+Z^1_t, \quad Y^i_t=S^{i\sm1}_{r,t}+Z^i_t \quad \forall t:t=nL+i, n\in\mathbb{N}, \qquad Y^i_t=0 \quad \text{otherwise}.
\end{align}
Here $Z^i_t\sim\mathcal{N}(0,N_i)$ denotes mutually independent white Gaussian noise components. Accordingly $\d$ receives $R_t=S^{L\sm1}_{r,t}+Z^L_t$ at $t=nL$ and zero otherwise.

\begin{figure}[!h]
    \centering
    \psfrag{p}[][][3.5]{\begin{sideways}Plant\end{sideways}}
    \psfrag{e}[][][3.5]{$\mathcal{E}$}
    \psfrag{d}[][][3.5]{$\mathcal{D}$}
    \psfrag{rn}[][][3.5]{$\mathcal{R}_L$}
    \psfrag{r1}[][][3.5]{$\mathcal{R}_1$}
    \psfrag{r2}[][][3.5]{$\mathcal{R}_2$}
    \psfrag{r3}[][][3.5]{$\mathcal{R}_{L\sm1}$}
    \psfrag{ch}[][][2]{AWGN Relay Channel}
    \psfrag{nfch}[][][2]{\begin{sideways}Noiseless Feedback Communication Channel\end{sideways}}
    \psfrag{z1}[][][2.5]{$Z^1_{t}$}
    \psfrag{z2}[][][2.5]{$Z^2_{t}$}
    \psfrag{z3}[][][2.5]{$Z^{L}_{t}$}
    \psfrag{st}[][][2.5]{$S_t$}
    \psfrag{rt}[][][2.5]{$R_t$}
    \psfrag{ut}[][][2.5]{$U_t$}
    \psfrag{s1}[][][2.5]{$S_{e,t}$}
    \psfrag{s2}[][][2.5]{$S^1_{r,t}$}
    \psfrag{s3}[][][2.5]{$S^2_{r,t}$}
    \psfrag{s4}[][][2.5]{$S^{L\sm1}_{r,t}$}
    \psfrag{y1}[][][2.5]{$Y^1_t$}
    \psfrag{y2}[][][2.5]{$Y^2_t$}
    \psfrag{y3}[][][2.5]{$Y^{L\sm1}_t$}
    \psfrag{y4}[][][2.5]{$R_t$}
    \resizebox{14 cm}{!}{\epsfbox{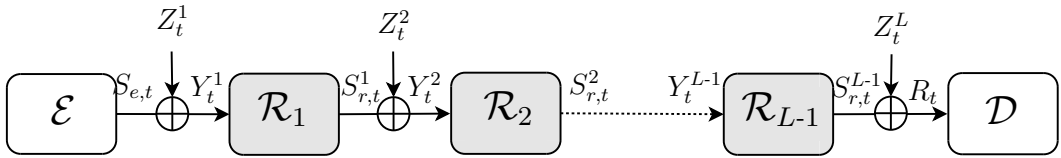}}
    \vspace{-.2cm}
  \caption{A cascade Gaussian network model.}\label{fig:CascadeNetwork}
\end{figure}

We now present a necessary condition for mean square stability over the given channel.
\begin{theorem}\label{thm:cascadeNetwork_nec}
If the system (\ref{eq:stateEquation}) is mean square stable over the \emph{cascade network} then
\begin{align}\label{eq:thmCascadenec}
\log\left(\left|\dt\left(A\right)\right|\right)<\frac{1}{2L} \log\left(1+L\min\left\{\frac{\ps}{\n_1},\frac{\pra}{\sum^{L}_{i=2}\n_i}\right\}\right).
\end{align}
\end{theorem}
\begin{proof}
We first derive an outer bound on the directed information $I(\bar{X}_{[1,LT]} \rightarrow R_{[1,LT]})$ over the given channel and then use Theorem \ref{thm:NecGeneral} to find the necessary condition \eqref{eq:thmCascadenec}.
\begin{align}\label{eq:boundNec1Cascade}
&I(\bar{X}_{[1,LT]} \rightarrow R_{[1,LT]})\stackrel{(a)}{=}I(\bar{X}_{[1,LT]} ; R_{[1,LT]})\stackrel{(b)}{\leq}I(\bar{X}_{[1,LT]}; Y^i_{[1,LT]}, R_{[1,LT]}) \nonumber \\
&= \sum^{LT}_{t=1}I(\bar{X}_{[1,LT]} ; R_t, Y^i_t | R_{[1,t\sm1]}, Y^i_{[1,t\sm 1]})\stackrel{(c)}{=} \sum^{LT}_{t=1} \Big( h(R_t, Y^i_t | R_{[1,t-1]}, Y^i_{[1,t\sm1]}) \nonumber \\
&\quad - h(R_t, Y^i_t | R_{[1,t-1]}, Y^i_{[1,t\sm 1]},\bar{X}_{[1,LT]})\Big)\stackrel{(d)}{=}\sum^{LT}_{t=1}\Big(h(Y^i_t | R_{[1,t-1]}, Y^i_{[1,t-1]}) + h(R_t|R_{[1,t\sm1]}, Y^i_{[1,t]})  \nonumber \\
&\quad - h(Y^i_t | R_{[1,t\sm1]}, Y^i_{[1,t\sm 1]},\bar{X}_{[1,LT]})- h(R_t|R_{[1,t\sm1]}, Y^i_{[1,t]},\bar{X}_{[1,LT]}) \Big) \nonumber \\
&\stackrel{(e)}{=}\sum^{LT}_{t=1}\Big(h(Y^i_t | R_{[1,t-1]}, Y^i_{[1,t-1]})- h(Y^i_t | R_{[1,t\sm1]}, Y^i_{[1,t\sm 1]},\bar{X}_{[1,LT]})+\underbrace{I(R_t;\bar{X}_{[1,LT]}|R_{[1,t\sm1]}, Y^i_{[1,t]})}_{=0}\Big) \nonumber \\
&\stackrel{(f)}{\leq}\sum^{LT}_{t=1}\left(h(Y^i_t)- h(Y^i_t |R_{[1,t\sm1]}, Y^i_{[1,t\sm 1]},\bar{X}_{[1,LT]})\right)\nonumber \\
&\stackrel{(g)}{\leq}\sum^{LT}_{t=1}\left(h(Y^i_t)- h(Y^i_t | S^{i\sm1}_{r,t},R_{[1,t\sm1]}, Y^i_{[1,t\sm 1]},\bar{X}_{[1,LT]})\right)\stackrel{(h)}{=}\sum^{LT}_{t=1}I(S^{i\sm1}_{r,t};Y^i_t) \nonumber\\
&\stackrel{(i)}{=}\sum^{T-1}_{t=0}I(S^{i\sm 1}_{r,tL+i};Y^i_{tL+i})\stackrel{(j)}{\leq}\frac{1}{2}\sum^{T-1}_{t=0}\log\left(1+\frac{L\pr^{i\sm 1}}{\n_i}\right)=\frac{T}{2}\log\left(1+\frac{L\pr^{i\sm 1}}{\n_i}\right)
\end{align}
where $(a)$ follows from \cite[Theorem 2]{Massey90}; $(b)$ follows from the fact that adding side information cannot decrease mutual information; $(c)$, $(d)$ and $(e)$ follow from properties of mutual information and differential entropy; $(f)$ follows from conditioning reduces entropy and the following Markov chain $\bar{X}_{[1,LT]}-(Y^i_{[1,t]},R_{[1,t\sm1]})-R_{t}$; $(g)$ follow from conditioning reduces entropy; $(h)$ follows from the Markov chain $Y^i_t-S^{i\sm1}_{r,t}-(R_{[1,t\sm1]}, Y^i_{[1,t\sm 1]},\bar{X}_{[1,LT]})$ due to memoryless channel from $S^{i\sm1}_{r,t}$ to  $Y^i_t$; $(i)$ follows from \eqref{eq:inOutCascadeRelay1} and \eqref{eq:inOutCascadeRelay2}; and $(j)$ follows from the fact that mutual information of a Gaussian channel is maximized by the Gaussian input distribution \cite[Theorem 8.6.5]{InfoBook}. If we replace $Y^i_{[1,LT]}$ with $Y^1_{[1,LT]}$ in step $(b)$ of \eqref{eq:boundNec1Cascade} and $S^{i\sm1}_{r,t}$ with $S_{e,t}$ in step $(g)$ of \eqref{eq:boundNec1Cascade}, then we get the following bound:
\begin{align}\label{eq:boundNec2Cascade}
I(\bar{X}_{[1,LT]} \rightarrow R_{[1,LT]})\leq\frac{T}{2}\log\left(1+\frac{L\ps}{\n_1}\right).
\end{align}
The directed information $I(\bar{X}_{[1,LT]} \rightarrow R_{[1,LT]})$ can also be bounded as
\begin{align}\label{eq:boundNec3Cascade}
&I\left(\bar{X}_{[1,LT]} \rightarrow R_{[1,LT]}\right)\stackrel{}{=}\sum^{LT}_{t=1}I\left(\bar{X}_{[1,t]} ; R_t | R_{[1,t\sm1]}\right)\stackrel{(a)}{\leq}\sum^{LT}_{t=1}I\left(S^{L\sm1}_{r,[1,t]} ; R_t| R_{[1,t\sm1]}\right)=I\left(S^{L\sm1}_{r,[1,LT]} \rightarrow R_{[1,LT]}\right) \nonumber \\
&\stackrel{(b)}{\leq}\sum^{LT}_{t=1}I(S^{L\sm 1}_{r,t};R_{t})\stackrel{(c)}{=}\sum^{T-1}_{t=0}I(S^{L\sm 1}_{r,tL+L};R_{tL+L})\stackrel{(d)}{\leq}\frac{T}{2}\log\left(1+\frac{L\pr^{L\sm 1}}{\n_L}\right),
\end{align}
where $(a)$ follows from the Markov chain $\bar{X}_{[1,LT]}-(S^{L\sm1}_{r,[1,t]},R_{[1,t\sm1]})-R_{[1,t]}$, $(b)$ follows from from \cite[Theorem 1]{Massey90}; $(c)$ follows from \eqref{eq:inOutCascadeRelay1} and \eqref{eq:inOutCascadeRelay2}; and $(d)$ follows from the fact that mutual information of a Gaussian channel is maximized by the Gaussian input distribution \cite[Theorem 8.6.5]{InfoBook}. Finally using \eqref{eq:boundNec1Cascade}, \eqref{eq:boundNec2Cascade}, and \eqref{eq:boundNec3Cascade}, we have the following bound:
\begin{align}\label{eq:directedInfoCascade}
&I(\bar{X}_{[1,LT]} \rightarrow R_{[1,LT]})\stackrel{}{\leq}\frac{T}{2}\min\left\{\log\left(1+\frac{L\ps}{\n_1}\right),\log\left(1+\frac{L\pr^1}{\n_2}\right),\dots,\log\left(1+\frac{L\pr^{L\sm1}}{\n_{L}}\right) \right\} \nonumber \\
&\stackrel{(a)}{=}\frac{T}{2}\log\left(1+L\min\left\{\frac{\ps}{\n_1},\frac{\pr^1}{\n_2},\dots,\frac{\pr^{L\sm1}}{\n_{L}}\right\}\right) \nonumber \\
&\leq\frac{T}{2}\log\left(1+L\min\left\{\frac{\ps}{\n_1},\max_{\pr^i: \sum \pr^i\leq \pra} \min \left\{\frac{\pr^1}{\n_2},\dots,\frac{\pr^{L\sm1}}{\n_{L}}\right\}\right\}\right) \nonumber \\
&\stackrel{(b)}{=}\frac{T}{2}\log\left(1+L \min\left\{\frac{\ps}{\n_1},\frac{\pra}{\sum^{L}_{i=2}\n_i}\right\}\right),
\end{align}
$(a)$ follows from the fact that $\log(1+x)$ is a monotonically increasing function of $x$; and $(b)$ follows from the optimal power allocation choice $\pr^i=\frac{\n_{i+1}\pra}{\sum^{L}_{i=2}\n_i}$. Finally dividing \eqref{eq:directedInfoCascade} by $LT$ and let $T\rightarrow \infty$ according to Theorem \ref{thm:NecGeneral}, we get the necessary condition \eqref{eq:thmCascadenec}.
\end{proof}

We now present a sufficient condition for mean-square stability over the given network.
\begin{theorem}\label{thm:cascadeNetwork_suff}
The scalar linear time invariant system in (\ref{eq:stateEquation}) with $A=\lm$ can be mean square stabilized using a linear scheme over a \emph{cascade network} of $L$ relay nodes if
\begin{align}\label{eq:thmCascadesuff}
\log\left(|\lm|\right)<\max_{\pr^i:\sum^L_{i=1}\pr^i \leq \pra} \frac{1}{2L} \log\left(1+\frac{L\ps}{L\ps+\n_1}\prod^{L-1}_{i=1}\left(\frac{L\pr^i}{L\pr^i+\n_{i+1}}\right)\right),
\end{align}
where the optimal power allocation is given by $\pr^i=\frac{-\n_{i+1}+\sqrt{\n^2_{i+1}-\frac{4\n_{i+1}}{\gamma}}}{2}$ and $\gamma<0$ is chosen such that $\sum^L_{i=1}\pr^i \leq \pra$. When all $N_i$ are equal, the optimal choice is $\pr^i=\frac{\pra}{L-1}$.
\end{theorem}
\emph{Outline of proof:}
The result can be derived by using a memoryless linear sensing and control scheme. Under linear policies, the overall mapping from the encoder to the controller becomes a scalar Gaussian channel, which has been well studied in the literature (see for example \cite{BansalBasar89}). Due to space constraints, we refer the reader to the proof of Theorem 5.2, which contains a detailed derivation for the non-orthogonal network and the proof for this setting is similar. The optimal power allocation follows from the concavity of $\prod^{L-1}_{i=1}\left(\frac{L\pr^i}{L\pr^i+\n_{i+1}}\right)$ in $\{\pr^i\}^{L-1}_{i=1}$ and by using the Lagrange multiplier method.

\begin{remark}
For fixed power allocations, as the number of relays $L$ approaches infinity in \eqref{eq:thmCascadenec}, the right hand side converges to zero and stabilization becomes impossible. We also note that the ratio between the sufficiency and necessity bounds converges to zero as the number of relays goes to infinity.
\end{remark}

In the related problem on the transmission of a Gaussian source with minimum mean-square distortion \cite{zaidiCDC11,LipsaMartinsAutomatica2011}, it is shown that linear sensing policies are not
globally optimal in general when there is one or more relay nodes in cascade. However linear policies
are shown to be person-by-person optimal in a single relay
setup. According to \cite{LipsaMartinsAutomatica2011,zaidiCDC11},
simple quantizer based policies can lead to a lower mean-square distortion than the best linear policy. We expect such non-linear policies to be useful for stabilization over cascade relay channels.

\section{Parallel Network}\label{sec:Parallel}

Consider the network shown in Fig.~\ref{fig:ParallelNetwork}, where the signal transmitted by a node does not interfere with the signals transmitted by other nodes, i.e., there are $L$ parallel channels from $\{\r_i\}^L_{i=1}$ to $\d$. We call this setup a \emph{parallel network}, which models a scenario where the signal spaces of the relay nodes are mutually orthogonal. For example the signals may be transmitted in either disjoint frequency bands or in disjoint time slots. In the first transmission phase, the sensor transmits $S_{e,t}$ with an average power $\ex\left[S^2_{e,t}\right]=2\ps$ to the relays and in the second phase all relays simultaneously transmit to the remote controller with average powers $2\pr^i$ such that $\sum^L_{i=1}\pr^i\leq\pra$. Accordingly, the received signals are given by
\begin{align}\label{eq:inOutCascadeRelay}
&Y^i_t=S_{e,t}+Z^i_{r,t}, \quad R^i_t=S^{i}_{r,t}=0, \quad t=1,3,5,\dots \nonumber\\
&R^i_t=S^{i}_{r,t}+Z^i_{d,t}, \quad Y^i_t=S_{e,t}=0, \quad t=2,4,6,\dots
\end{align}
where $Z^i_{r,t}\sim\mathcal{N}(0,N^i_r)$, $Z^i_{d,t}\sim\mathcal{N}(0,N^i_d)$  denote mutually independent white Gaussian noise variables.
In the following we present conditions for mean square stability of the system in \eqref{eq:stateEquation} over the given parallel network.

\begin{figure}[!tb]
    \centering
    \psfrag{p}[][][3]{\begin{sideways}Plant\end{sideways}}
    \psfrag{e}[][][3]{$\mathcal{E}$}
    \psfrag{d}[][][3]{$\mathcal{D}$}
    \psfrag{rn}[][][3]{$\mathcal{R}_L$}
    \psfrag{r1}[][][3]{$\mathcal{R}_1$}
    \psfrag{r2}[][][3]{$\mathcal{R}_2$}
    \psfrag{r3}[][][3]{$\mathcal{R}_3$}
    \psfrag{ch}[][][2]{AWGN Relay Channel}
    \psfrag{nfch}[][][2]{\begin{sideways}Noiseless Feedback Communication Channel\end{sideways}}
    \psfrag{z}[][][2]{$Z_{t}$}
    \psfrag{z1}[][][2]{$Z^1_{r,t}$}
    \psfrag{z2}[][][2]{$Z^2_{r,t}$}
    \psfrag{zn}[][][2]{$Z^L_{r,t}$}
    \psfrag{v1}[][][2]{$Z^1_{d,t}$}
    \psfrag{v2}[][][2]{$Z^2_{d,t}$}
    \psfrag{vl}[][][2]{$Z^L_{d,t}$}
    \psfrag{y1}[][][2]{$Y^1_{t}$}
    \psfrag{y2}[][][2]{$Y^2_{t}$}
    \psfrag{ym}[][][2]{$Y^L_{t}$}
    \psfrag{y}[][][2]{$Y_{t}$}
    \psfrag{se}[][][2]{$S_{e,t}$}
    \psfrag{sr1}[][][2]{$S^1_{r,t}$}
    \psfrag{sr2}[][][2]{$S^2_{r,t}$}
    \psfrag{sr3}[][][2]{$S^L_{r,t}$}
    \psfrag{h1}[][][2]{$h_1$}
    \psfrag{h2}[][][2]{$h_2$}
    \psfrag{hn}[][][2]{$h_L$}
    \psfrag{h}[][][2]{$h$}
    \psfrag{st}[][][2.5]{$S_t$}
    \psfrag{rr1}[][][2]{$R^1_{t}$}
    \psfrag{rr2}[][][2]{$R^2_{t}$}
    \psfrag{rrn}[][][2]{$R^L_{t}$}
    \psfrag{ut}[][][2.5]{$U_t$}
    \resizebox{8 cm}{!}{\epsfbox{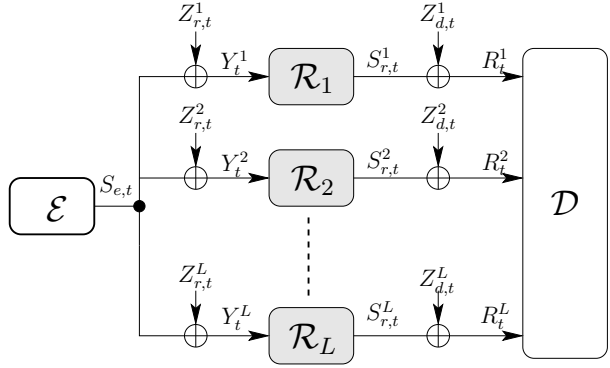}}
    \vspace{-.3cm}
  \caption{Parallel relay network.}\label{fig:ParallelNetwork}
\end{figure}

\begin{theorem}\label{thm:parallelNec}
If the system (\ref{eq:stateEquation}) is mean square stable over the \emph{parallel network} then
\begin{align}\label{eq:thmParallel_nec}
\log\left(\left|\dt\left(A\right)\right|\right)\leq \frac{1}{4}\min\left\{\log\left(1+2\sum^L_{i=1}\frac{\ps}{\nr^i}\right),\sum^L_{i=1}\log\left(1+\frac{2\pr^i}{\nd^i}\right) \right\},
\end{align}
where $\pr^i=\max\{\gamma - \nd^i, 0\}$ and $\gamma$ is chosen such that $\sum^L_{i=1}\pr^i=\pra$.
\end{theorem}

\begin{proof}
Following the same steps as in proof of Theorem \ref{thm:cascadeNetwork_nec}, we can bound directed information $I(\bar{X}_{[1,2T]} \rightarrow R_{[1,2T]})$ over \emph{parallel} relay network as,
\begin{align}\label{eq:directedInfoParallel}
&I(\bar{X}_{[1,2T]} \rightarrow \{R^i_{[1,2T]}\}^L_{i=1}) \stackrel{(a)}{\leq}\min\left\{\sum^{2T}_{t=1}I\left(S_{e,t};\{Y^i_t\}^L_{i=1}\right), \sum^{2T}_{t=1}I\left(\{S^i_{r,t}\}^L_{i=1};\{R^i_t\}^L_{i=1}\right)\right\}  \nonumber\\
&\stackrel{(b)}{=}\min\left\{\sum^{T}_{t=1}I\left(S_{e,2t\sm 1};\{Y^i_{2t\sm1}\}^L_{i=1}\right), \sum^{T}_{t=1}I\left(\{S^i_{r,2t}\}^L_{i=1};\{R^i_{2t}\}^L_{i=1}\right)\right\} \nonumber \\
&\stackrel{(c)}{\leq}\frac{T}{2}\min\left\{\log\left(1+2\sum^L_{i=1}\frac{\ps}{\nr^i}\right),\max_{\pr^i: \pr\geq0,\sum_i\pr^i\leq\pra}\sum^L_{i=1}\log\left(1+\frac{2\pr^i}{\nd^i}\right) \right\},
\end{align}
where $(a)$ follows from the same steps as in \eqref{eq:boundNec1Cascade} and \eqref{eq:boundNec3Cascade}; $(b)$ follows from \eqref{eq:inOutCascadeRelay}; and $(c)$ follows from the fact that Gaussian input distribution maximizes mutual information for a Gaussian channel. The function $\sum^L_{i=1}\log\left(1+\frac{2\pr^i}{\nd^i}\right)$ is jointly concave in $\{\pr^i\}^L_{i=1}$. The optimal power allocation is given by $\pr^i=\max\{\gamma - \nd^i/2, 0\}$, where $\gamma$ is chosen such that $\sum^L_{i=1}\pr^i=\pra$, which is the well-known water-filling solution \cite[pp. 204-205]{TseBook}. We obtain \eqref{eq:thmParallel_nec} by using \eqref{eq:directedInfoParallel} in Theorem \ref{thm:NecGeneral}.
\end{proof}
We can obtain a sufficient condition for mean square stability over the \emph{parallel network} using linear policies like previously discussed scenarios, which is stated in the following theorem.
\begin{theorem}\label{thm:parallel_suff}
The scalar linear time invariant system in (\ref{eq:stateEquation}) with $A=\lm$ can be mean square stabilized using a linear scheme over the Gaussian \emph{parallel network} if
\begin{align}\label{eq:thmParallel_suff}
\log\left(\left|\lambda\right|\right)<\frac{1}{4} \log\left(1+\sum^L_{i=1}\frac{4\ps \pr^i}{2\ps \nd + 2\pr^i \nr^i +\nd \nr^i}\right).
\end{align}
\end{theorem}
\begin{proof}
The above result can be obtained by using a memoryless linear sensing and control scheme and as discussed in the proof of Theorem \ref{thm:cascadeNetwork_suff} and Theorem \ref{thm:HalfDup}. 
\end{proof}

\begin{proposition}
The gap between the necessary and sufficient conditions for a symmetric parallel network with $\pr^i=\pr, \nr^i=\nr$ is a non-decreasing function of the number of relays $L$ and approaches $\frac{1}{4}\log\left(1+\frac{\nd\left(2\ps+\nr\right)}{2\pr\nr}\right)$ as $L$ goes to infinity.
\end{proposition}
\begin{proof}
For $\pr^i=\pr, \nr^i=\nr$, the R.H.S. of \eqref{eq:thmParallel_suff} is evaluated as $\Gamma_{\text{suf}}:=\frac{1}{4}\log\left(1+\frac{4L\ps\pr}{2\ps\nd+2\pr\nr+\nd\nr}\right)$ and the R.H.S of \eqref{eq:thmParallel_nec} can be bounded as $\Gamma_{\text{nec}}:=\frac{1}{4}\log\left(1+\frac{2L\ps}{\nr}\right)$. The gap is given by
\begin{align}
\Gamma_{\text{nec}}-\Gamma_{\text{suf}}=\frac{1}{4}\log\left(1+ \frac{2\ps \nd\left(2\ps+\nr\right)}{4\ps\pr\nr+\frac{\nr\left(2\ps\nd+2\pr\nr+\nd\nr\right)}{L}}\right),
\end{align}
which is an increasing function of $L$, approaching $\frac{1}{4}\log\left(1+\frac{\nd\left(2\ps+\nr\right)}{2\pr\nr}\right)$ as $L\rightarrow\infty$.
\end{proof}
\begin{remark}
If $N^i_d=0$, then $\Gamma_{\text{nec}}-\Gamma_{\text{suf}}=0$ and the linear scheme is exactly optimal. For $\nr^i=0$, $\Gamma_{\text{suf}}:=\frac{1}{4}\log\left(1+\frac{2L\pr}{\nd}\right)$ and $\Gamma_{\text{nec}}:=\frac{L}{4}\log\left(1+\frac{2\pr}{\nd}\right)$ according to \eqref{eq:thmParallel_nec}. Clearly $\lim_{L \rightarrow \infty}\left(\Gamma_{\text{nec}}-\Gamma_{\text{suf}}\right)=\infty$, showing the inefficiency of the LTI scheme for parallel channels.
\end{remark}

It is known that
linear schemes can be sub-optimal for transmission over parallel channels \cite{VaishampayanThesis,WernerssonSkoglund09}.
A distributed joint source--channel code is optimal in
minimizing mean-square distortion if the following two conditions hold
\cite{ShamaiZamir98}: i) All channels from the source to the
destination send independent information; ii) All channels utilize the
capacity, i.e., the source and channel need to be matched. If we use
linear policies at the relay nodes then the first condition is not
fulfilled because all nodes would be transmitting correlated
information. In \cite{YukselTatikonda09} the authors proposed a
non-linear scheme for a parallel network of two sensors without
relays, in which one sensor transmits only the magnitude of the
observed state and the other sensor transmits only the phase of the
observed state. The magnitude and phase of the state are shown to be
independent and thus the scheme fulfills the first condition of
optimality. This nonlinear sensing scheme is shown to outperform the
best linear scheme for the LQG control problem in the absence of
measurement noise, although the second condition of source-channel
matching is not fulfilled. We can use this non-linear scheme together
with the initialization step of the Schalkwijk Kailath (SK) type scheme described in Appendix \ref{sec:ProofHalfDuplex} for the non-orthogonal network,
which will ensure source-channel matching by
making the outputs of the two sensors Gaussian distributed after the
initial transmissions. In \cite{MattiasZaidi11} it is shown that linear sensing policies may not
be even person-by-person optimal for LQG control over parallel network without
relays.

For the special case of \emph{parallel network} with noiseless $\e-\r_i$ links, we have the following necessary and sufficient condition for mean-square stability.
\begin{theorem}\label{thm:parallelNoiseless}
The system (\ref{eq:stateEquation}) in absence of process noise ($W_t=0$) can be mean square stabilized over the Gaussian \emph{parallel network} with $Z^i_{r,t}=0$ for all $i$, only if
\begin{align}
\log\left(\left|\dt\left(A\right)\right|\right)\leq\frac{1}{4}\max_{\pr^i: \pr\geq0,\sum_i\pr^i\leq\pra}\sum^L_{i=1}\log\left(1+\frac{2\pr^i}{\nd^i}\right).
\end{align}
If the inequality is strict, then there exists a non-linear policy leading to mean-square stability.
\end{theorem}
\begin{proof}
The necessity follows from Theorem \ref{thm:parallelNec}. The sufficiency part for scalar systems follows from \cite[Theorem 6]{KumarLaneman11}, which is derived using a non-linear scheme. This scheme can be extended to vector systems using a time sharing scheme presented in Sec.~\ref{sec:multiDimension}.
\end{proof}

\begin{remark}\label{rem:parallelRelay}
According to Theorem \ref{thm:parallelNoiseless} the minimum rate required for stabilization of a noisy plant over a parallel Gaussian channel is equal to the channel capacity. It was shown by Shu and Middleton in \cite{ShuMiddleton11} that for some first order noiseless plants, linear time invariant encoders/decoders cannot achieve this minimum rate over parallel Gaussian channels. However the minimum rate for stabilization can always be achieved by a non-linear time varying scheme as discussed in the proof of Theorem \ref{thm:parallelNoiseless}.
\end{remark}

\section{Non-orthogonal Network}\label{sec:NonOrthogonal}
A communication network is said to be \emph{non-orthogonal} if all the communicating nodes transmit signals in overlapping time slots using the same frequency bands. A node which is capable of transmitting and receiving signals simultaneously using the same frequency band is known as \emph{full-duplex} while a \emph{half-duplex} node cannot simultaneously receive and transmit signals. In practice it is expensive and hard to a build a communication device which can transmit and receive signals at the same time using the same frequency, due to the self-interference created by the transmitted signal to the received signal. Therefore half-duplex systems are mostly used in practice. In this section we study both half-duplex and full-duplex configurations.

\subsection{Non-orthogonal Half-duplex Network}\label{sec:HalfDuplex}
\begin{figure}[!t]
    \centering
    \subfigure[First transmission phase.]{
    \psfrag{p}[][][3]{\begin{sideways}Plant\end{sideways}}
    \psfrag{e}[][][3]{$\mathcal{E}$}
    \psfrag{d}[][][3]{$\mathcal{D}$}
    \psfrag{rn}[][][3]{$\mathcal{R}_L$}
    \psfrag{r1}[][][3]{$\mathcal{R}_1$}
    \psfrag{r2}[][][3]{$\mathcal{R}_2$}
    \psfrag{r3}[][][3]{$\mathcal{R}_3$}
    \psfrag{ch}[][][2]{AWGN Relay Channel}
    \psfrag{nfch}[][][2]{\begin{sideways}Noiseless Feedback Communication Channel\end{sideways}}
    \psfrag{z}[][][2]{$Z_{d,t}$}
    \psfrag{z1}[][][2]{$Z^1_{r,t}$}
    \psfrag{z2}[][][2]{$Z^2_{r,t}$}
    \psfrag{zn}[][][2]{$Z^L_{r,t}$}
    \psfrag{y1}[][][2]{$Y^1_{t}$}
    \psfrag{y2}[][][2]{$Y^2_{t}$}
    \psfrag{ym}[][][2]{$Y^L_{t}$}
    \psfrag{y}[][][2]{$Y_{t}$}
    \psfrag{se}[][][2]{$S_{e,t}$}
    \psfrag{sr1}[][][2]{$S^1_{r,t}$}
    \psfrag{sr2}[][][2]{$S^2_{r,t}$}
    \psfrag{sr3}[][][2]{$S^L_{r,t}$}
    \psfrag{h1}[][][2]{$h_1$}
    \psfrag{h2}[][][2]{$h_2$}
    \psfrag{hn}[][][2]{$h_L$}
    \psfrag{h}[][][2]{$h$}
    \psfrag{st}[][][2.5]{$S_t$}
    \psfrag{r}[][][2]{$R_t$}
    \psfrag{ut}[][][2.5]{$U_t$}
    \resizebox{7.7 cm}{!}{\epsfbox{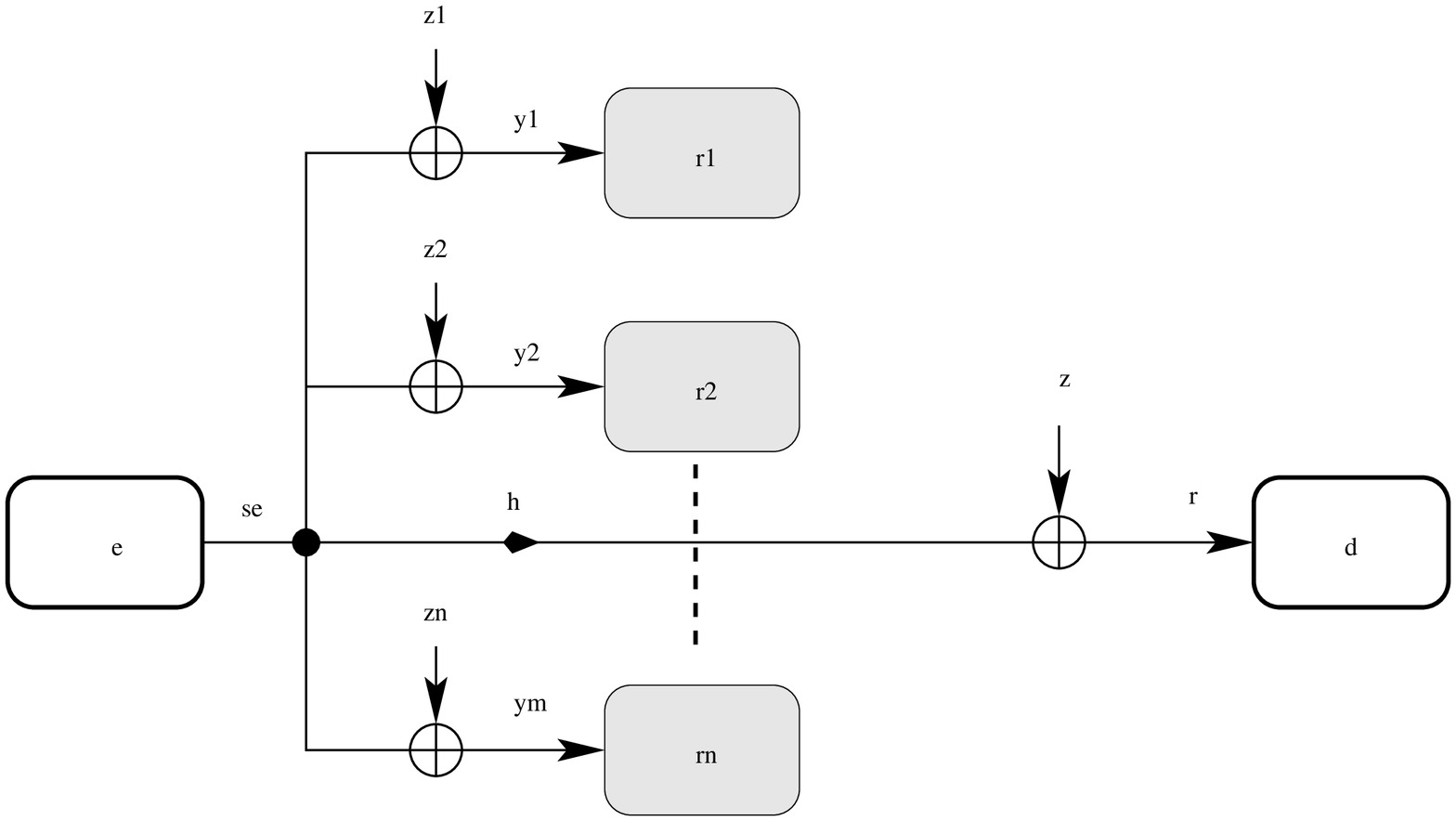}}
    \label{fig:phaseOne}}
    \vspace{1cm}
    \subfigure[Second transmission phase.]{
    \centering
    \psfrag{p}[][][3]{\begin{sideways}Plant\end{sideways}}
    \psfrag{e}[][][3]{$\mathcal{E}$}
    \psfrag{d}[][][3]{$\mathcal{D}$}
    \psfrag{rn}[][][3]{$\mathcal{R}_L$}
    \psfrag{r1}[][][3]{$\mathcal{R}_1$}
    \psfrag{r2}[][][3]{$\mathcal{R}_2$}
    \psfrag{r3}[][][3]{$\mathcal{R}_3$}
    \psfrag{ch}[][][2]{AWGN Relay Channel}
    \psfrag{nfch}[][][2]{\begin{sideways}Noiseless Feedback Communication Channel\end{sideways}}
    \psfrag{z}[][][2]{$Z_{d,t}$}
    \psfrag{z1}[][][2]{$Z^1_{r,t}$}
    \psfrag{z2}[][][2]{$Z^2_{r,t}$}
    \psfrag{zn}[][][2]{$Z^L_{r,t}$}
    \psfrag{y1}[][][2]{$Y^1_{t}$}
    \psfrag{y2}[][][2]{$Y^2_{t}$}
    \psfrag{ym}[][][2]{$Y^L_{t}$}
    \psfrag{y}[][][2]{$Y_{t}$}
    \psfrag{se}[][][2]{$S_{e,t}$}
    \psfrag{sr1}[][][2]{$S^1_{r,t}$}
    \psfrag{sr2}[][][2]{$S^2_{r,t}$}
    \psfrag{sr3}[][][2]{$S^L_{r,t}$}
    \psfrag{h1}[][][2]{$h_1$}
    \psfrag{h2}[][][2]{$h_2$}
    \psfrag{hn}[][][2]{$h_L$}
    \psfrag{h}[][][2]{$h$}
    \psfrag{st}[][][2.5]{$S_t$}
    \psfrag{r}[][][2]{$R_t$}
    \psfrag{ut}[][][2.5]{$U_t$}
    \resizebox{7.8 cm}{!}{\epsfbox{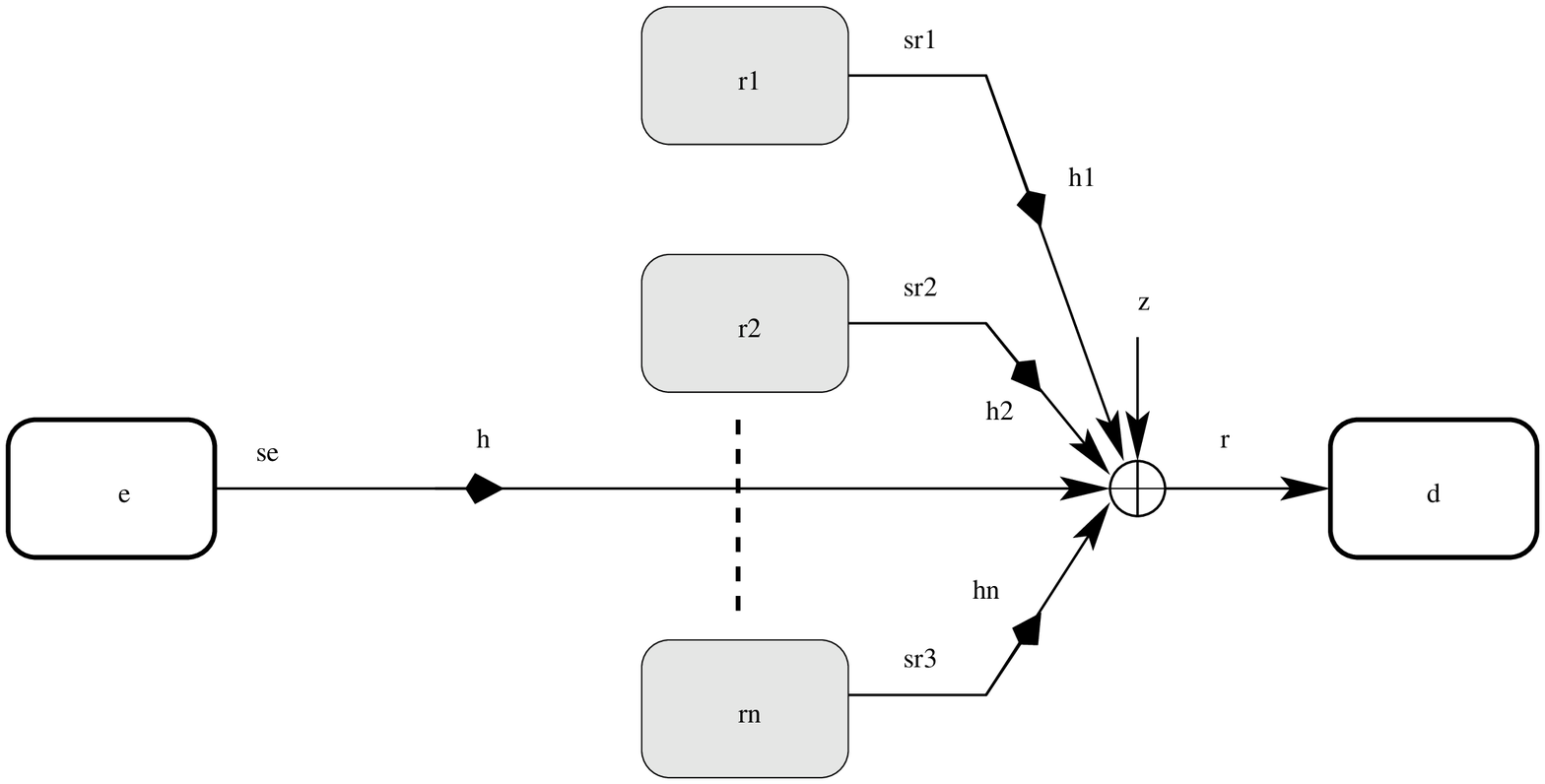}}
    \label{fig:phaseTwo}}
    \vspace{-.8cm}
   \caption{A non-orthogonal half-duplex Gaussian network model.}
   \label{fig:HalfDuplexNetwork}
 \vspace{5mm}
\end{figure}

A non-orthogonal half-duplex Gaussian network with $L$ relay nodes
$\{\r_i\}^L_{i=1}$ is illustrated in
Fig.~\ref{fig:HalfDuplexNetwork}. The variables $S_{e,t}$ and
$S^i_{r,t}$ denote the transmitted signals from the state encoder $\e$
and relay $\r_i$ at any discrete time step $t$. The variables
$Z^i_{r,t}$ and $Z_{d,t}$ denote the mutually independent white
Gaussian noise components at the relay node $i$ and $\d$ of the remote
control unit, with $Z^i_{r,t}\sim\mathcal{N}(0,N^i_r)$ and
$Z_{d,t}\sim\mathcal{N}(0,N_d)$. The noise components
$\{Z^i_{r,t}\}^L_{i=1}$ are independent across the relays, i.e.,
$\ex[Z^k_{r,t}Z^i_{r,t}]=0$ for all $i\neq k$. The information
transmission from the state encoder consists of two phases as shown in
Fig.~\ref{fig:HalfDuplexNetwork}. In the first phase the encoder
$\mathcal{E}$ transmits a signal with an average power $2\beta \ps$,
where $0<\beta\leq1$ is a parameter that adjusts power between the two
transmission phases. In this transmission phase all the relay nodes
listen but remain silent. In the second transmission phase, the
encoder $\e$ and relay nodes $\{\r_i\}^L_{i=1}$ transmit
simultaneously. In this second transmission phase, the encoder
transmits with an average power $2(1-\b)\ps$ and the $i$-th relay node
transmits with an average power $2\pr^i$ such that
$\sum^L_{i=1}\pr^i\leq P_R$. The input and output of the $i$-th relay
are given by,
\begin{align}\label{eq:halfDupRelay_InOutEq}
&Y^i_t=S_{e,t}+Z^i_{r,t}, \quad S^i_{r,t}=0, \qquad &t=1,3,5,\dots \nonumber \\
&Y^i_t=0, \quad S^i_{r,t}=g^i_t\left(Y^i_{[0,t-1]}\right), \qquad &t=2,4,6,\dots
\end{align}
where $g^i_t: \mathbb{R}^{t+1}\rightarrow\mathbb{R}$ is the $i$-th relay encoding policy such that $\ex\left[\left(g^i_t(Y^i_{[0,t-1]})\right)^2\right]=2\pr^i$ and $\sum^L_{i=1}\pr^i\leq P_R$.
The signal received at the decoder/controller is given by
\begin{align}
R_t&=h S_{e,t}+\sum^L_{i=1} h_iS^i_{r,t}+Z_{d,t}, \nonumber
\end{align}
where $h,h_i\in\mathbb{R}$ denote the channel gains of $\e-\d$ and $\r_i-\d$ links respectively.

\begin{theorem}\label{thm:NonOrthHalf_Nec}
If the linear system in \eqref{eq:stateEquation} is mean-square stable over the
\emph{non-orthogonal half-duplex} relay network, then
\begin{align}\label{eq:thmNonOrthHalf_Nec}
&\log\left(|\dt\left(A\right)|\right)\leq \frac{1}{4}\min\Bigg\{\max_{\begin{subarray}{c} 0<\b\leq 1\end{subarray}}\left( \log\left(1+\frac{2h^2(1-\b)\ps}{\nd}\right)+\log\left(1+2\b\ps\left(\sum^L_{i=1}\frac{1}{\nr^i}+\frac{h^2}{N_d}\right)\right)\right), \nonumber \\
&\max_{\begin{subarray}{c} 0<\b\leq1 \\ \pr^i: \sum_i \pr^i \leq\pra \end{subarray}}\left( \log\left(1\!+\!\frac{2h^2\b\ps}{\nd}\right)\!+\! \log\left(1\!+\!\frac{1}{\nd}\left(\!\sum^{L+1}_{i=1}\!\delta^2_iP_i\!+\!2\!\sum^{L+1}_{i=1}\sum^{L+1}_{k=i+1}\!\rho^\star_{i,k}\delta_i\delta_k\sqrt{P_i P_k}\right)\right)\right)\Bigg\},
\end{align}
where $\rho^\star_{i,k}:=\frac{2(1-\b)\ps}{\sqrt{(2(1-\b)\ps+N_i)(2(1-\b)\ps+N_k)}}$, $P_{L+1}:=2(1-\b)\ps$, $N_{L+1}:=0$, $\delta_{L+1}:=h$, $P_i:=2\pr^i$, $\delta_i:=h_i$, $N_i:=\nr^i$ for all $i=\{1,2,\dots,L\}$.
\end{theorem}
\begin{proof}
We first derive an outer bound on the directed information $I(\bar{X}_{[1,LT]} \rightarrow R_{[1,LT]})$ over the given channel and then use Theorem \ref{thm:NecGeneral} to find the necessary condition \eqref{eq:thmNonOrthHalf_Nec}.
\begin{align}\label{eq:nonOrthHalf_Bnd1}
&I(\bar{X}_{[1,2T]} \rightarrow R_{[1,2T]})\stackrel{(a)}{=}I\left(\bar{X}_{[1,2T]} ; R_{[1,2T]}\right)\stackrel{(b)}{\leq}I(\bar{X}_{[1,2T]}; \{Y^i_{[1,2T]}\}^L_{i=1}, R_{[1,2T]}) \nonumber \\
&\stackrel{(c)}{=}I\left(\bar{X}_{[1,2T]}; \tilde{R}_{[1,2T]}, \{Y^i_{[1,2T]}\}^L_{i=1}\right)= \sum^{2T}_{t=1}I(\bar{X}_{[1,2T]} ; \tilde{R}_t, \{Y^i_{t}\}^L_{i=1} | \tilde{R}_{[1,t\sm1]}, \{Y^i_{[1,t\sm1]}\}^L_{i=1}) \nonumber \\
&\stackrel{(d)}{\leq}\sum^{2T}_{t=1}I(S_{e,t} ; \tilde{R}_t, \{Y^i_{t}\}^L_{i=1} | \tilde{R}_{[1,t\sm1]}, \{Y^i_{[1,t\sm1]}\}^L_{i=1})\stackrel{(e)}{\leq} \sum^{2T}_{t=1}I\left(S_{e,t}; \tilde{R}_{t}, \{Y^i_{t}\}^L_{i=1}\right) \nonumber \\
&\stackrel{(f)}{=}\sum^{T}_{t=1}I\left(S_{e,2t}; \tilde{R}_{2t}\right)+\sum^{T}_{t=1}I\left(S_{e,2t\sm 1}; \tilde{R}_{2t\sm 1},\{Y^i_{2t\sm 1}\}^L_{i=1}\right)  \nonumber \\
&\stackrel{(g)}{\leq} \frac{T}{2} \log\left(1+\frac{2h^2(1-\b)\ps}{\nd}\right) +\frac{T}{2}\log\left(1+2\b\ps\left(\sum^L_{i=1}\frac{1}{\nr^i}+\frac{h^2}{N_d}\right)\right) \nonumber \\
& \stackrel{}{\leq} \frac{T}{2} \max_{\begin{subarray}{c} 0<\b\leq 1\end{subarray}}\left\{ \log\left(1+\frac{2h^2(1-\b)\ps}{\nd}\right)+\log\left(1+2\b\ps\left(\sum^L_{i=1}\frac{1}{\nr^i}+\frac{h^2}{N_d}\right)\right)\right\}
\end{align}
where $(a)$ follows from \cite[Theorem 1]{Massey90}; $(b)$ follows from the fact that adding side information cannot decrease mutual information; $(c)$ follows by defining $\tilde{R}_{t}:=R_t-\sum^L_{i=1} h_iS^i_{r,t}$ and from the fact that $S^i_{r,t}$ is a function of $Y^i_{[1,t\sm1]}$; $(d)$ follows from the Markov chain $\bar{X}_{[1,2T]}-S_{e,t}-(\tilde{R}_{t},\{Y^i_{t}\}^L_{i=1})$, since $\bar{X}_{[0,T]}$ is the uncontrolled state process and the fact that the channel between $S_{e,[1,2T]}$ and $(\tilde{R}_{[1,2T]}, \{Y^i_{[1,2T]}\}^L_{i=1})$ is memoryless due to $\tilde{R}_{t}=R_t-\sum^L_{i=1} h_iS^i_{r,t}$; $(e)$ follows from the Markov chain $(\tilde{R}_{[1,t\sm1]}, \{Y^i_{[1,t\sm1]}\}^L_{i=1})-S_{e,t}-(\tilde{R}_{t}, \{Y^i_{t}\}^L_{i=1})$ and conditioning reduces entropy; $(f)$ follows by separating odd and even indexed terms and $Y^i_{2t}=0$ according to \eqref{eq:halfDupRelay_InOutEq}; $(g)$ follows from $Y^i_{2t\sm1}=S_{e,2t\sm1}+Z^i_{r,2t\sm 1}$, $\tilde{R}_t=S_{e,t}+Z_t$, $\ex\left[S^2_{e,2t}\right]=2(1-\b)\ps$, $\ex\left[S^2_{e,2t\sm1}\right]=2\b\ps$, and the fact that mutual information of a Gaussian channel is maximized by centered Gaussian input distribution \cite{TseBook}. The directed information rate $I(\bar{X}_{[1,2T]} \rightarrow R_{[1,2T]})$ can also be bounded as,
\begin{align}\label{eq:nonOrthHalf_Bnd2}
&I(\bar{X}_{[1,2T]} \rightarrow R_{[1,2T]})=\sum^{2T}_{t=1}I(\bar{X}_{[1,t]} ; R_t| R_{[1,t\sm1]}) \stackrel{(a)}{\leq}\sum^{2T}_{t=1}I(S_{e,t},\{S^i_{r,t}\}^L_{i=1} ; R_t| R_{[1,t\sm1]})\nonumber \\
&\stackrel{(b)}{\leq} \sum^{2T}_{t=1}I\left(S_{e,t},\{S^i_{r,t}\}^L_{i=1}; R_{t}\right)\stackrel{(c)}{=}\sum^{T}_{t=1}I\left(S_{e,2t\sm 1}; R_{2t\sm 1}\right)+ \sum^{T}_{t=1}I\left(S_{e,2t},\{S^i_{r,2t}\}^L_{i=1}; R_{2t}\right)\nonumber \\
&\stackrel{(d)}{\leq} \frac{T}{2} \log\left(1+\frac{2h^2\b\ps}{\nd}\right)+ \frac{T}{2}\log\left(1\!+\!\frac{1}{\nd}\left(\!\sum^{L+1}_{i=1}\!\delta^2_iP_i\!+\!2\!\sum^{L+1}_{i=1}\sum^{L+1}_{k=i+1}\!\rho^\star_{i,k}\delta_i\delta_k\sqrt{P_i P_k}\right)\right)\nonumber \\
&\stackrel{}{\leq} \frac{T}{2} \max_{\begin{subarray}{c} 0<\b\leq1 \\ \pr^i: \sum_i \pr^i \leq\pra \end{subarray}}\left\{ \log\left(1+\frac{2h^2\b\ps}{\nd}\right)+ \log\left(1\!+\!\frac{1}{\nd}\left(\!\sum^{L+1}_{i=1}\!\delta^2_iP_i\!+\!2\!\sum^{L+1}_{i=1}\sum^{L+1}_{k=i+1}\!\rho^\star_{i,k}\delta_i\delta_k\sqrt{P_i P_k}\right)\right)\right\}
\end{align}
where $\rho^\star_{i,k}=\frac{2(1-\b)\ps}{\sqrt{(2(1-\b)\ps+N_i)(2(1-\b)\ps+N_k)}}$, $P_{L+1}=2(1-\b)\ps$, $N_{L+1}=0$, $\delta_{L+1}=h$, $P_i=2\pr^i$, $\delta_i=h_i$, $N_i=\nr^i$ for all $i=\{1,2,\dots,L\}$. The inequality $(a)$ follows from the Markov chain $\bar{X}_{[0,t]}-\left(S_{e,t},\{S^i_{r,t}\}^L_{i=1}\right)-R_t$ due to the memoryless channel between $S_{e,[1,2T]},\{S^i_{r,[1,2T]}\}^L_{i=1}$ and ${R}_{[1,2T]}$; $(b)$ follows from the Markov chain $R_{[1,t\sm1]}-\left(S_{e,t},\{S^i_{r,t}\}^L_{i=1}\right)-R_t$ and conditioning reduces entropy; $(c)$ follows by separating the odd and even indexed terms and $S^i_{r,2t\sm 1}=0$ according to \eqref{eq:halfDupRelay_InOutEq}; $(d)$ follows from the fact that the first addend on the R.H.S. of $(c)$ is maximized by a centered Gaussian distributed $S_{e,t}$ and the second addend is bounded using a bound presented in \cite{GastparITA07}, where the author studied the problem of transmitting a Gaussian source over a simple sensor network. In order to apply the upper bound given in (48) of \cite{GastparITA07} to our setup, we consider state encoder $\e$ to be a sensor node with zero observation noise and make the following change of system variables so that our system model becomes equivalent to the one discussed in \cite{GastparITA07}: $\sigma^2_S:=\a_t$, $\delta_i:=h_i$, $M:=L+1$, $P_i:=2\pr^i$, $\sigma^2_Z:=\nd$, $\sigma^2_{W,i}:=\nr^i$, $\a_i=\sqrt{\frac{2(1-\b)\ps}{\a_t}}$ for all $i$. We finally obtain \eqref{eq:thmNonOrthHalf_Nec} by dividing \eqref{eq:nonOrthHalf_Bnd1} and \eqref{eq:nonOrthHalf_Bnd2} by $2T$ and let $T\rightarrow \infty$ according to Theorem \ref{thm:NecGeneral}.
\end{proof}

We now present a sufficient condition for mean square stability of a scalar plant over the given network, which can be extended to a multi-dimensional plant using the arguments given in Sec.~\ref{sec:multiDimension}.
\begin{theorem}\label{thm:HalfDup} The scalar linear time invariant system in (\ref{eq:stateEquation}) with $A=\lm$ can be mean square stabilized using a linear scheme over the \emph{non-orthogonal half-duplex} network if
\begin{align}
\log\left(|\lm|\right)\!<\!\frac{1}{4}\max_{\begin{subarray}{c} 0<\b\leq1 \\ \pr^i: \sum_i \pr^i \leq\pra \end{subarray}}\left\{\log\left(1+\frac{2h^2\b\ps}{\nd}\right)+\log\left(1+\frac{\tilde{M}\left(\b,\left\{\pr^i\right\}^L_{i=1}\right)}{\nt\left(\b,\left\{\pr^i\right\}^L_{i=1}\right)}\right)\right\},
\label{eq:thmHalfDup}
\end{align}
where $\tilde{M}\left(\b,\left\{\pr^i\right\}^L_{i=1}\right)=\left(\sqrt{2h^2(1-\b)\ps}+\sqrt{\frac{2\b\ps\nd}{\left(2h^2\b\ps+\nd\right)}}
\left(\sum^L_{i=1}\sqrt{\frac{2h^2_i\pr^i}{2\b\ps+\nr^i}}\right)\right)^2$ and $\nt(\b,\left\{\pr^i\right\}^L_{i=1})=\sum^L_{i=1}\frac{2h^2_i\pr^i\nr^i}{2\b\ps+\nr^i}+\nd$ are real-valued functions.
\end{theorem}
\begin{proof}
The proof is given in Appendix \ref{sec:ProofHalfDuplex}.
\end{proof}

\begin{remark}
An optimal choice of the power allocation parameter $\b$ at the state encoder and an optimal power allocation at the relay nodes $\{\pr^i\}^L_{i=1}$ which maximize the term on the right hand side of (\ref{eq:thmHalfDup}) depend on the quality of the $\e-\d$, $\e-\r_i$, and $\r_i-\d$ links. This is a non-convex optimization problem, however it can be transformed into an equivalent convex problem by using the approach in \cite[Appendix A]{XiaoGoldsmith08}. This equivalent convex problem can be efficiently solved for optimal $\{\pr^i\}^L_{i=1}$ using the interior point method. For $\b=1$, we can analytically obtain the following optimal power allocation using the Lagrangian method:
\begin{align} \label{eq:optimalPowerNonOrth}
\pr^i=\pra\left(\frac{h^2_i\left(2\ps+\nr^i\right)}{\left(2\ps\nd+\nr^i\nd+\pra h^2_i \nr^i\right)^2}\right)\left[\sum^L_{l=1}\frac{h^2_l\left(2\ps+\nr^l\right)}{\left(2\ps\nd+\nr^l\nd+\pra h^2_l \nr^l\right)^2}\right]^{-1}.
\end{align}
\end{remark}
\begin{remark}\label{rem:infoRate}
  For channels with feedback, directed information is a useful
  quantity \cite{Massey90,TatikondaMitter09}. It is shown in
  Appendix~\ref{appendixAchievability} that the term on the right hand
  side of (\ref{eq:thmHalfDup}) is the information rate over the
  half-duplex network with noiseless feedback, obtained when running
  the described closed-loop protocol. Further we show that the directed
  information rate is also equal to the term on the right hand side of
  (\ref{eq:thmHalfDup}).
\end{remark}

\subsection{Two-Hop Network}
Consider the half-duplex relay network illustrated in Fig.~\ref{fig:HalfDuplexNetwork} with $h=0$. The state information is communicated to the remote controller only via the relay nodes. We call this setup a \emph{two-hop} relay network, where the communication from the state encoder to the controller takes place in two hops. In the first hop the relay nodes receive the state information from the state encoder, which then communicate the state information to the controller in the second hop. The controller takes action in alternate time steps upon receiving the state information. We can obtain a sufficient condition for stability over this network by substituting $h=0, \b=1$ in Theorem \ref{thm:HalfDup}. Similarly a necessary condition can be obtained from \eqref{eq:thmNonOrthHalf_Nec}, where $\b=1$ is the maximizer of the first term and $\b=0$ is the maximizer of the second term. In the following we evaluate the gap between the sufficient and necessary conditions for a symmetric two hop network.

\begin{proposition}
For a symmetric two-hop network with $\pr^i=\pr, \nr^i=\nr, h_i=c, h=0,\b=1$, the gap between necessary and sufficient conditions approaches zero as the number of relays $L$ goes to infinity. The gap also monotonically approaches zero as $\pr$ goes to infinity.
\end{proposition}
\begin{proof}
For $\pr^i=\pr, \nr^i=\nr, h_i=c, h=0,\b=1$ for all $i$, the R.H.S. of \eqref{eq:thmHalfDup} is evaluated as $\Gamma_{\text{suf}}:=\frac{1}{4}\log\left(1+\frac{4L^2c^2\ps\pr}{2Lc^2\pr\nr+\nd\left(2\ps+\nr\right)}\right)$ and the R.H.S of \eqref{eq:thmNonOrthHalf_Nec} can be bounded as $\Gamma_{\text{nec}}:=\frac{1}{4}\log\left(1+\frac{2L\ps}{\nr}\right)$. The gap between $\Gamma_{\text{suf}}$ and $\Gamma_{\text{nec}}$ is given by
\begin{align}
\Gamma_{\text{nec}}-\Gamma_{\text{suf}}=\frac{1}{4}\log\left(1+ \frac{\frac{4\ps^2\nd+2\ps\nr\nd}{L}}{4c^2\ps\pr\nr+\frac{2c^2\pr\nr^2}{L}+\frac{\nd\nr\left(2\ps+\nr\right)}{L^2}}\right),
\end{align}
which approaches zero as $L$ goes to infinity. The gap also monotonically approaches zero as $\pr$ tends to infinity.
\end{proof}
In Fig.~\ref{fig:nonOrthCompLPr} we have plotted $\Gamma_{\text{nec}}$ and $\Gamma_{\text{suf}}$ as functions of $L$ and $\pr$. These figures show that linear schemes are quite efficient in some regimes.

\begin{remark}
Linear policies can be even exactly optimal in the following special cases: i) If we fix all relaying policies to be linear, then the channel becomes equivalent to a point-point scalar Gaussian channel, for which linear sensing is known to be optimal for LQG control \cite{BansalBasar89}. ii) If we fix the state encoder to be linear and assume noiseless causal feedback links from the controller to the relay nodes, then linear policies are optimal for mean-square stabilization over a symmetric \emph{two-hop} relay network, by the following arguments. Since the control actions are available at the relay nodes via noiseless feedback links, there is no dual effect of control, i.e., the separation of estimation and control holds. Further by restricting the state encoder to be linear, the relay network becomes equivalent to the Gaussian network studied in\cite{GastparIT08,GastparITA07}, where it is shown that linear policies are optimal if the network is symmetric.
\end{remark}

\begin{figure}\label{fig:nonOrthCompLPr}
\centering
\subfigure[$\ps=2\pr^i=10,\nr^i=\nd=1,h_i=1$]{\includegraphics[width=7cm]{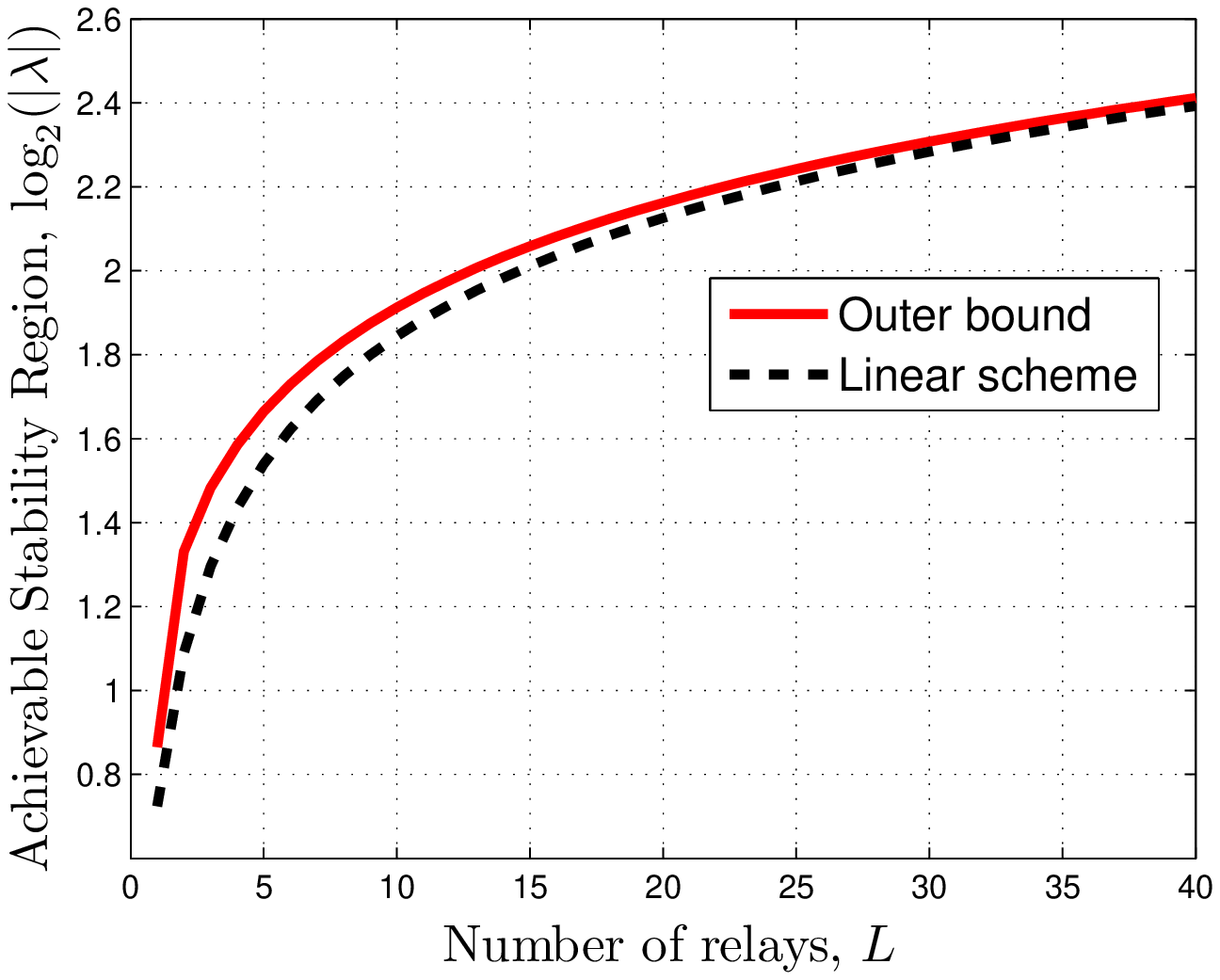}}
\subfigure[$L=10,\ps=10,\nr^i=\nd=1,h_i=1$]{\includegraphics[width=7cm]{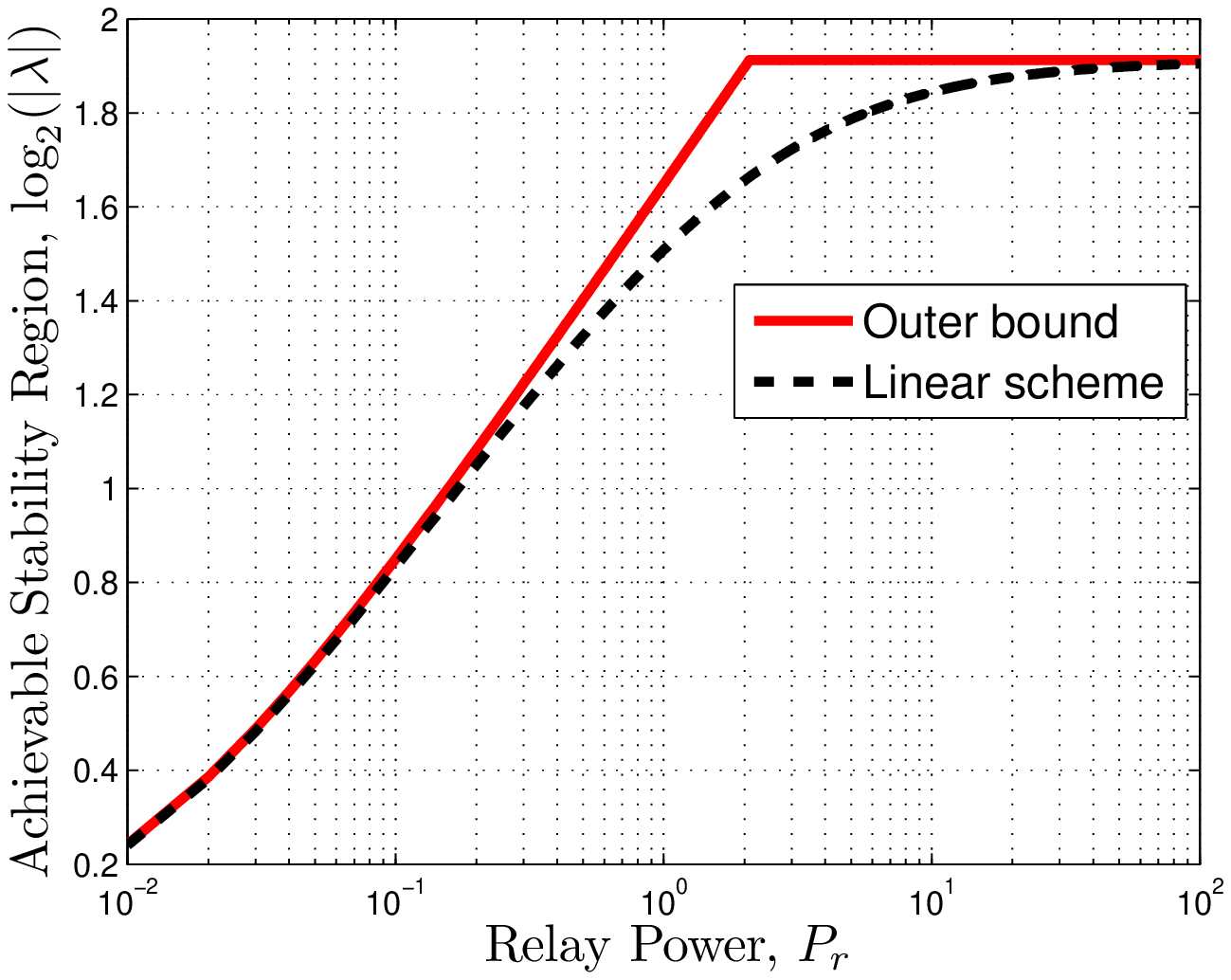}}
\caption{Comparison of necessary and sufficient conditions for a symmetric two-hop relay network.}
\label{subfig2}
\end{figure}

\subsection{Non-orthogonal Full-duplex  Network}\label{sec:FullDuplex}

We now consider a non-orthogonal network of $L$ \emph{full-duplex} relay nodes, where all the nodes receive and transmit their signals in every time step, i.e., at any time instant $t\in\mathbb{N}$,
\begin{align}
&S_{e,t}=f_t\left(X_{[0,t]},U_{[0,t-1]}\right), \qquad S^i_{r,t}=g^i_t\left(Y^i_{[0,t-1]}\right), \quad \forall t\in\mathbb{N},\nonumber \\
&Y^i_t=S_{e,t}+Z^i_{r,t}, \qquad R_t=hS_{e,t}+\sum^L_{i=1}S^i_{r,t}+Z_{d,t}, \quad \forall t\in\mathbb{N},
\end{align}
where $\ex\left[(S_{e,t})^2\right]=\ps$, $\ex\left[(S^i_{r,t})^2\right]=\pr^i$, and $\sum^L_{i=1}\pr^i\leq P_R$.
%


\begin{theorem}\label{thm:NonOrthFull_Nec}
If the linear system in \eqref{eq:stateEquation} is mean-square stable over the
\emph{non-orthogonal full-duplex} relay network, then
\begin{align}\label{eq:thmNonOrthFull_Nec}
\log\left(|\dt\left(A\right)|\right)\leq& \frac{1}{2}\min\Bigg\{\log\left(1+\ps\left(\sum^L_{i=1}\frac{1}{\nr^i}+\frac{h^2}{N_d}\right)\right), \nonumber \\
&\max_{\begin{subarray}{c} \pr^i: \sum_i \pr^i \leq\pra \end{subarray}}\left( \log\left(1\!+\!\frac{1}{\nd}\left(\!\sum^{L+1}_{i=1}\!\delta^2_iP_i\!+\!2\!\sum^{L+1}_{i=1}\sum^{L+1}_{k=i+1}\!\rho^\star_{i,k}\delta_i\delta_k\sqrt{P_i P_k}\right)\right)\right)\Bigg\},
\end{align}
where $\rho^\star_{i,k}=\frac{\ps}{\sqrt{(\ps+N_i)(\ps+N_k)}}$, $P_{L+1}:=\ps$, $N_{L+1}:=0$, $\delta_{L+1}:=h$, $P_i:=\pr^i$, $\delta_i:=h_i$, $N_i:=\nr^i$ for all $i=\{1,2,\dots,L\}$.
\end{theorem}
\begin{proof}
  The proof follows exactly in the steps of the proof of Theorem
  \ref{thm:NonOrthHalf_Nec}, with an exception that odd and even
  indexed terms are not treated separately because
  $\ex\left[S^2_{e,t}\right]=\ps$ and
  $\ex\left[(S^i_{r,t})^2\right]=\pr^i$ for all $t$.
\end{proof}

\begin{theorem}\label{thm:FullDup}
The scalar linear time invariant system in (\ref{eq:stateEquation}) with $A=\lm$ and $W_t=0$ can be mean square stabilized using a linear scheme over the \emph{non-orthogonal full-duplex} Gaussian network if
\begin{align}\label{eq:thmFullDup}
\log\left(|\lambda|\right)\!<\!\frac{1}{2}\max_{\pr^i:\sum^L_{i=1}\pr^i \leq \pra} \left\{\log \left(1\!+\!{\left(\sqrt{h^2\ps}\!+\!\eta^\star \sum^L_{i=1}\sqrt{\frac{h^2_i\ps\pr^i}{\ps\!+\!\nr^i}}\right)^2}\left({\nd\!+\!\sum^L_{i=1}\frac{h^2_i\pr^i \nr^i}{\ps\!+\!\nr^i}}\right)^{\!-\!1}\right)\right\},
\end{align}
where $\eta^\star$ is the unique root in the interval $[0,1]$ of the following fourth order polynomial

\begin{align}\label{eq:chan3Poly}
\left(\sum^L_{i=1}\sqrt{\frac{h^2_i\ps\pr^i}{(\ps+\nr^i)}}\right)\eta^4&+\left(2h\ps\sum^L_{i=1}\sqrt{\frac{h^2_i\pr^i}{(\ps+\nr^i)}}\right)\eta^3 \nonumber \\
&+\left(h^2\ps+\nd+\sum^L_{i=1}\frac{h^2_i\pr^i \nr^i}{\ps+\nr^i}\right)\eta^2=\left(\nd+\sum^L_{i=1}\frac{h^2_i\pr^i \nr^i}{\ps+\nr^i}\right).
\end{align}

\end{theorem}
\begin{proof}
The proof can be found in \cite{zaidiICCA10} for a single relay setup, which can be easily extended for multiple relays.
\end{proof}
Although we expect that Theorem \ref{thm:FullDup} also holds in the presence of process noise like other setups, we are not able to show convergence of second moment of the state process. However numerical experiments suggest that the result should hold.

\begin{remark} The term on the right hand side of the inequality in (\ref{eq:thmFullDup}) is an achievable rate with which information can be transmitted reliably over the \emph{non-orthogonal full-duplex} relay network. This result is derived for a network with single relay node in \cite[Theorem 5]{BrossWigger09}, however it can be easily extended to problems with multiple relays.
\end{remark}

\section{Noisy Multi-dimensional Systems}\label{sec:multiDimension}
In this section we investigate stabilization of multi-dimensional systems over multi-dimensional channels. First we state a result for a scalar Gaussian channel.
\begin{theorem}\label{thm:multiDim}
The $n$-dimensional noisy linear system \eqref{eq:stateEquation} can be mean square stabilized over a scalar Gaussian channel having information capacity $\mathcal{C}$, if $\log\left(\left|A\right|\right)<\mathcal{C}$. Furthermore, a linear time varying policy is sufficient through sequential linear encoding
of scalar components.
\end{theorem}
\emph{Proof Outline:} We prove Theorem \ref{thm:multiDim} with the help of a simple example, due to space limitation in the paper.
Consider that a two-dimensional plant with system matrix $A=\left[
  \begin{array}{cc}
    \lm_1 & 1 \\
    0 & \lm_2 \\
  \end{array}
\right]$
and an invertible input matrix $B$ has to be stabilized over a Gaussian channel disturbed by a zero mean Gaussian noise with variance $N$.
We assume that the sensor transmits with an average $P$. For this channel, we define information
capacity as $\mathcal{C}:=\frac{1}{2}\log\left(1+\frac{P}{N}\right)$.
We denote the state and the control variables as
$X_t:=[x_{1,t},x_{2,t}]^T$ and $U_t:=[u_{1,t},u_{2,t}]^T$
respectively. Consider the following scheme for stabilization. The
sensor observes state vector $X_t$ in alternate time steps (that is,
at $t, t+2, t+4, \dots$), whose elements are sequentially
transmitted. The sensor linearly transmits $x_{2,t}$ at time $t$ and
$x_{1,t}$ at time $t+1$ with an average transmit power constraint. The
control actions for the two modes are also taken in alternate time
steps, that is, $u_{1,t}=0$ and $u_{2,t+1}=0$. Accordingly the state
equations for the two modes at time $t+1$ are given by
\begin{align}
x_{2,t+1}&=\lm_2 x_{2,t}+ u_{2,t}+w_{2,t}\stackrel{(a)}{=}\lm_2 \left(x_{2,t}-\hat{x}_{2,t}\right)+w_{2,t}, \label{eq:vectorMode2t1}\\
x_{1,t+1}&\stackrel{(b)}{=}\lm_1 x_{1,t}+ x_{2,t} +w_{1,t}, \label{eq:vectorMode1t1}
\end{align}
where $(a)$ and $(b)$ follow from $u_{2,t}=-\lm_2 \hat{x}_{2,t}$ and $u_{1,t}=0$. The state equations at time $t+2$ are
\begin{align}
x_{2,t+2}&=\lm_2 x_{2,t+1}+w_{2,t+1}=\lm^2_2 \left(x_{2,t}-\hat{x}_{2,t}\right)+\lm_2 w_{2,t}+ w_{2,t+1}, \label{eq:vectorMode2t2}\\
x_{1,t+2}&=\lm_1 x_{1,t+1}\!+\! x_{2,t+1} \!+\! u_{1,t+1} \!+\! w_{1,t+1}\stackrel{(a)}{=}\lm^2_1 x_{1,t}\!+\! (\lm_1\!+\!\lm_2) x_{2,t} \!+\! u_{1,t+1} \!+\! \lm_1 w_{1,t} \!+\! w_{2,t} \!+\! w_{1,t+1}  \nonumber \\
&\stackrel{(b)}{=}\lm^2_1\left(x_{1,t}- \hat{x}_{1,t}\right)+ (\lm_1+\lm_2)\left(x_{2,t}-\hat{x}_{2,t}\right) + \lm_1 w_{1,t} +w_{2,t}+ w_{1,t+1}, \label{eq:vectorMode1t2}
\end{align}
where $(a)$ follows \eqref{eq:vectorMode1t1}; and $(b)$ follows from $u_{1,t+1}=\!-\!\lm^2_1 \hat{x}_{1,t}\! -\! (\lm_1 \!+\!\lm_2) \hat{x}_{2,t}$. We first study the stabilization of the lower mode. According to \eqref{eq:vectorMode2t2} the second moment of $x_{2,t}$ is given by
\begin{align}
\ex\left[x^2_{2,t+2}\right]=\lambda^4_2 \ex\left[\left(x_{2,t}- \hat{x}_{2,t}\right)^2\right]+ \tilde{n}_2= \lambda^4_2 2^{-2\mathcal{C}} \ex\left[x^2_{2,t}\right]+\tilde{n}_2.
\end{align}
where the last equality follows from the linear mean-square estimation of a Gaussian variable over a scalar Gaussian channel of capacity $\mathcal{C}$ and $\tilde{n}_2:=(\lm^2_2+1) n_{w,2}$. We observe that the lower mode is stable if and only if $\lambda^4_2 2^{-2\mathcal{C}}<1 \Rightarrow \log(|\lm_2|) < \frac{\mathcal{C}}{2}$.
Assuming that $x_{2,t}$ is stable, the second moment of $x_{1,t}$ is given by
\begin{align}
&\ex\left[x^2_{1,t+2}\right]\stackrel{(a)}{=}\lambda^4_1 \ex\left[\left(x_{1,t}- \hat{x}_{1,t}\right)^2\right]+ 2\lambda^2_1 (\lm_1+\lm_2) \ex\left[\left(x_{1,t}- \hat{x}_{1,t}\right)\left(x_{2,t}-\hat{x}_{2,t}\right)\right] \nonumber \\
&\quad +(\lm_1+\lm_2)^2 \ex\left[\left(x_{2,t}-\hat{x}_{2,t}\right)^2\right] + \tilde{n}_1 \nonumber \\
&\stackrel{(b)}{=}\lambda^4_1 2^{-2\mathcal{C}} \ex\left[x^2_{1,t}\right]+ 2\lambda^2_1 (\lm_1+\lm_2)  \ex\left[\left(x_{1,t}- \hat{x}_{1,t}\right)\left(x_{2,t}-\hat{x}_{2,t}\right)\right]+(\lm_1+\lm_2)^2 2^{-2\mathcal{C}} \ex\left[x^2_{2,t}\right] + \tilde{n}_1 \nonumber \\
&\stackrel{(c)}{\leq}\lambda^4_1 2^{-2\mathcal{C}} \ex\left[x^2_{1,t}\right]\!+\! 2\lambda^2_1 (\lm_1\!+\!\lm_2) \sqrt{\ex\left[\left(x_{1,t}\!-\! \hat{x}_{1,t}\right)^2\right] \ex\left[\left(x_{2,t}\!-\!\hat{x}_{2,t}\right)^2\right]}\!+\!(\lm_1\!+\!\lm_2)^2 2^{-2\mathcal{C}} \ex\left[x^2_{2,t}\right] \!+\! \tilde{n}_1 \nonumber \\
&=\lambda^4_1 2^{-2\mathcal{C}} \ex\left[x^2_{1,t}\right]+ 2\lambda^2_1 (\lm_1+\lm_2) \sqrt{2^{-2\mathcal{C}} \ex\left[x^2_{1,t}\right]}\sqrt{2^{-2C}\ex\left[x^2_{2,t}\right]}+(\lm_1+\lm_2)^2 2^{-2\mathcal{C}} \ex\left[x^2_{2,t}\right] + \tilde{n}_1 \nonumber \\
&\stackrel{(d)}{\leq} k_1 \ex\left[x^2_{1,t}\right]+ k_2\sqrt{\ex\left[x^2_{1,t}\right]}+k_3,
\end{align}
where $(a)$ follows from \eqref{eq:vectorMode1t2} and
$\tilde{n}_1:=(\lm^2_1+1) n_{w,1}+n_{w,2}$; $(b)$ follows from the
linear mean-square estimation of a Gaussian variable over a scalar
Gaussian channel of capacity $\mathcal{C}$; $(c)$ follows from the
Cauchy--Schwarz inequality; $(d)$ follows from the fact
$\ex\left[x^2_{1,t}\right]<M$ (assuming that $\lambda^4_2
2^{-2\mathcal{C}}<1$) and by defining $k_1:=\lambda^4_1 2^{-2\mathcal{C}}$,
$k_2:=2\lambda^2_1 (\lm_1+\lm_2) 2^{-2\mathcal{C}} \sqrt{M}$, and
$k_3:=(\lm_1+\lm_2)^2 2^{-2\mathcal{C}} M +\tilde{n}_1$. We now want to a find
condition which ensures convergence of the following sequence:
\begin{align}
\a_{t+1}=k_1 \a_{t}+ k_2\sqrt{\a_{t}}+k_3. \label{eq:vectorSeq1}
\end{align}
In order to show convergence, we make use of the following lemma.
\begin{lemma}\label{lm:vectorConverge}
Let $T:\mathbb{R}\mapsto\mathbb{R}$ be a non-decreasing continuous mapping with a unique fixed point $x^\star\in\mathbb{R}$. If there exists $u\leq x^\star \leq v$ such that $T(u)\geq u$ and $T(v) \leq v$, then the sequence generated by $x_{t+1}=T(x_t)$, $t\in\mathbb{N}$ converges starting from any initial value $x_0\in \mathbb{R}$.
\end{lemma}
\begin{proof}
The proof is given in Appendix \ref{apx:MultiDim}.
\end{proof}
We observe that the mapping $T(\a)=k_1 \a+ k_2\sqrt{\a}+k_3$ with $\a\geq0$ is monotonically increasing since $k_1,k_2>0$. It will have a unique fixed point $\a^\star$ if and only if $k_1<1$, since $k_2,k_3>0$. Assuming that $k_1<1$, there exists $u<\a^\star<v$ such that $T(u)\geq u$ and $T(v) \leq v$. Therefore by Lemma \ref{lm:vectorConverge} the sequence $\{\a_t\}$ is convergent if $k_1=\lambda^4_1 2^{-2\mathcal{C}}<1 \Rightarrow \log(|\lm_1|) < \frac{\mathcal{C}}{2}$.

The time sharing scheme illustrated above can be generalized to any
$n$-dimensional plant and the stability conditions can be easily
obtained using Lemma \ref{lm:vectorConverge}. We know that any system
matrix $A$ can be written in the Jordan form by a similarity matrix
transformation. We can then use the following scheme for
stabilization. The encoder chooses to send only one component of the
observed state vector at each time $t$ over a Gaussian channel of
capacity $\mathcal{C}$. Assume that for a fraction
$\frac{\log(|\lm_m|)}{\sum^K_{i=1}\log(|\lm_i|)}$ of the total
available time the encoder transmits the $m$-th component of the state
vector. Thus the rate available for the transmission of the $m$-th
state component is
$\frac{\log(|\lm_m|)}{\sum^K_{i=1}\log(|\lm_i|)}\mathcal{C}$. The
system will be stable if and only if
$\log(|\lm_m|)<\frac{\log(|\lm_m|)}{\sum^K_{i=1}\log(|\lm_i|)}\mathcal{C}$
for all $m\in \{1,2,\dots,K\}$, which implies
$\sum^K_{i=1}\log(|\lm_i|)=\log\left(\left|\dt\left(A\right)\right|\right)<
\mathcal{C}$. For a multi-dimensional system with a controllable $(A,B)$ pair, any input (control action) can be realized in $n$ time steps. If the encoder has access to the channel output, then it can refine estimate of the state using noiseless feedback channel (SK coding scheme) during these $n$ time steps and observe the new state periodically after every $n$ time steps. \hfill $\square$
\begin{remark}The sufficiency results
presented in sections \ref{sec:NonOrthogonal}-\ref{sec:Parallel} for
scalar systems can be extended to multi-dimensional systems using the proposed time varying scheme. The
sufficient conditions for vector systems will be identical to scalar
systems except that $\log(|\lm|)$ is replaced with
$\log(\left|\dt\left(A\right)\right|)$ everywhere.
\end{remark}
\begin{remark}
  In \cite{BraslavskyFreudenberg07} the authors studied stabilization
  of a noiseless multi-dimensional system over a point-to-point scalar
  Gaussian channel using a linear time invariant scheme, that is the
  state encoder transmits $S_t=E X_t$, where $E$ is a row vector. This
  LTI scheme cannot stabilize if the pair $(A,E)$ is not
  observable. For example consider a diagonal system matrix $A$ with
  two equal eigenvalues. This system cannot be stabilized by any
  choice of the encoding matrix $E$, irrespective of how much power
  the state encoder is allowed to spend. However our linear time varying scheme can always stabilize
  the system, even in the presence of process noise.
\end{remark}

\begin{remark}
As mentioned in Theorem \ref{thm:parallelNoiseless} and Remark \ref{rem:parallelRelay}, the proposed time varying scheme can be used with the non-linear scheme of \cite{KumarLaneman11} to achieve the minimum power required for stabilization of noiseless multi-dimensional plants over vector Gaussian channels.
\end{remark}

\vspace{-.3cm}
\section{Conclusions}
The problem of mean-square stabilization of LTI plants over basic Gaussian relay networks is analyzed. Some necessary and sufficient conditions for stabilization are presented which reveal relationships between stabilizability and communication parameters. These results can serve as a useful guideline for a system designer. Necessary conditions have been derived using information theoretic cut-set bounds, which are not tight in general due to the real-time nature of the information transmission. Sufficient conditions for stabilization of scalar plants are obtained by employing time invariant communication and control schemes. We have shown that time invariant schemes are not sufficient in general for stabilization of multi-dimensional plants. However, a simple time variant scheme is always shown to stabilize multi-dimensional plants. In this time varying scheme, one component of the state vector is transmitted at a time and the state component corresponding to a more unstable mode is transmitted more often. The sufficient conditions for stabilization of multi-dimensional plants are obtained by using this time varying scheme. We also established minimum signal-to-noise ratio requirement for stabilization of a noiseless multi-dimensional plant over a parallel Gaussian channel. It is observed in some network settings that sufficient conditions do not depend on the plant noise and they may be characterized by the directed information rate from the sequence of channel inputs to the sequence of channel outputs. We have discussed optimality of linear policies over the given network topologies. In some very special cases, linear schemes are shown to be optimal.

\vspace{-.3cm}
\appendix
\subsection{Necessary Condition}\label{apx:NecGeneral}
Consider the following series of inequalities:
\begin{align}\label{eq:directedInfoBound1}
&I\left(X_{[0,T\sm1]}\rightarrow R_{[0,t\sm 1]}\right)\stackrel{(a)}{=}\sum^{T-1}_{t=0}I\left(X_{[0,t]};R_t|R_{[0,t\sm 1]}\right)\stackrel{(b)}{\geq}\sum^{T-1}_{t=0}I\left(X_t;R_t|R_{[0,t\sm 1]}\right)\nonumber \\
&\stackrel{}{=}I\left(X_0;R_0\right)\!+\!\sum^{T-1}_{t=1}I\left(X_t;R_t|R_{[0,t\sm 1]}\right)\!=\! I\left(X_0;R_0\right)+ \sum^{T-1}_{t=1}\left(h\left(X_t|R_{[0,t\sm 1]}\right)\!-\!h\left(X_t|R_{[0,t]}\right)\right)\nonumber \\
&\stackrel{(c)}{=}I\left(X_0;R_0\right)+\sum^{T-1}_{t=1}\left(h\left(AX_{t-1}+BU_{t-1}+W_{t-1}|R_{[0,t\sm 1]}\right)-h\left(X_t|R_{[0,t]}\right)\right) \nonumber \\
&\stackrel{(d)}{=}I\left(X_0;R_0\right)+\sum^{T-1}_{t=1}\left(h\left(AX_{t-1}+W_{t-1}|R_{[0,t\sm 1]}\right)-h\left(X_t|R_{[0,t]}\right)\right) \nonumber \\
&\!\stackrel{(e)}{\geq}\!I\left(X_0;R_0\right)+\sum^{T-1}_{t=1}\left(h\left(AX_{t-1}+W_{t-1}|R_{[0,t\sm 1]},W_{t-1}\right)-h\left(X_t|R_{[0,t]}\right)\right)\stackrel{(f)}{=}I\left(X_0;R_0\right)+ \nonumber \\
&+\sum^{T-1}_{t=1}\left(h\left(AX_{t-1}|R_{[0,t\sm 1]}\right)-h\left(X_t|R_{[0,t]}\right)\right)\stackrel{(g)}{=}\sum^{T-1}_{t=1}\left(\log\left(\left|\dt\left(A\right)\right|\right)+h\left(X_{t-1}|R_{[0,t\sm 1]}\right)-h\left(X_t|R_{[0,t]}\right)\right) \nonumber\\
&+I\left(X_0;R_0\right)\stackrel{}{=}I\left(X_0;R_0\right)+T \log\left(\left|\dt\left(A\right)\right|\right)+h\left(X_0|R_0\right)-h\left(X_{T\sm1}|R_{[0,T\sm 1]}\right)\nonumber \\
&\stackrel{}{=}h\left(X_0\right)+ T \log\left(\left|\dt\left(A\right)\right|\right)-h\left(X_{T\sm1}|R_{[0,T\sm 1]}\right) \stackrel{(h)}{\geq} h\left(X_0\right)+ T \log\left(\left|\dt\left(A\right)\right|\right)-h\left(X_{T\sm1}\right)\nonumber\\
& \stackrel{(i)}{\geq} h\left(X_0\right)+ T \log\left(\left|\dt\left(A\right)\right|\right)-\log\left((2\pi e)^n\left|\dt\left(K\right)\right|\right),
\end{align}
where $(a)$ follows from the definition of directed information; $(b)$ follows from the fact that discarding variables cannot increase mutual information; $(c)$ follows from \eqref{eq:stateEquation}; $(d)$ follows from $U_{t- 1}=\pi_{t-1}(R_{[0,t\sm 1]})$; $(e)$ follows from the fact that conditioning reduces entropy; $(f)$ follows from $h\left(AX_{t\sm1}+W_{t\sm1}|R_{[0,t\sm 1]},W_{t-1}\right)=h(AX_{t-1}|R_{[0,t\sm 1]},W_{t-1})=h(AX_{t-1}|R_{[0,t\sm 1]})$ due to mutual independence of $X_t$ and $W_t$; $(g)$ follows from $h\left(AX\right)=\left|\dt\left(A\right)\right| h(X)$ \cite[Theorem 8.6.4]{InfoBook}; $(h)$ follows from conditioning reduces entropy; and $(i)$ follows the fact that for a mean square stable system there exists a matrix $K\succ0$ with $\dt\left(\ex\left[X^T_tX_t\right]\right) < \dt\left(K\right)$ for all $t$ and further for a given covariance matrix the differential entropy is maximized by the Gaussian distribution.
We can also write
\begin{align}\label{eq:noControlDirect}
&I\left(X_{[0,T\sm1]}\rightarrow R_{[0,t\sm 1]}\right)\stackrel{}{=}\sum^{T-1}_{t=0}I\left(X_{[0,t]};R_t|R_{[0,t\sm 1]}\right)\stackrel{(a)}{=}\sum^{T-1}_{t=0}I\left(\bar{X}_{[0,t]}+\bar{f}\left(U_{[0,t\sm1]}\right);R_t|R_{[0,t\sm 1]}\right) \nonumber \\
&\stackrel{(b)}{=}\sum^{T-1}_{t=0}I\left(\bar{X}_{[0,t]};R_t|R_{[0,t\sm 1]}\right)=I\left(\bar{X}_{[0,T\sm1]}\rightarrow R_{[0,T\sm 1]}\right),
\end{align}
where $(a)$ follows by defining uncontrolled state process $\bar{X}_{t+1}=A\bar{X}_{t}+W_t$ and writing the controlled state process $X_{[0,t]}$ as a sum of uncontrolled process $\bar{X}_{[0,t]}$ and a linear function of control actions $U_{[0,t\sm1]}$, since the system is linear and control actions are additive; and $(b)$ follows from $U_{t}=\pi_{t}(R_{[0,t]})$. From \eqref{eq:directedInfoBound1} and \eqref{eq:noControlDirect} we have  ${\liminf}_{\begin{subarray}{c}T\rightarrow\infty\end{subarray}} \frac{1}{T} I\left(\bar{X}_{[0,T\sm1]}\rightarrow R_{[0,T\sm 1]}\right){\geq}\log\left(\left|\dt\left(A\right)\right|\right)$, since we have assumed $h(X_0)<\infty$.

\subsection{Proof of Theorem \ref{thm:HalfDup}}\label{sec:ProofHalfDuplex}
In order to prove Theorem \ref{thm:HalfDup} we propose a linear
communication and control scheme. This scheme is based on the coding
scheme given in \cite{BrossWigger09} which is an adaptation of the
well-known Schalkwijk--Kailath scheme \cite{SchalkwijkKailath66}. By
employing the proposed linear scheme, we find a condition on the
system parameters $\lm$ which is sufficient to mean square stabilize
the system (\ref{eq:stateEquation}). The control and communication
scheme for the \emph{half-duplex} network works as follows: If the
initial state $X_0$ is not Gaussian distributed, then we first make
the state process Gaussian distributed by performing the following
initialization step which was introduced in \cite{zaidiReglermote}.
\subsubsection*{Initial time step, $t=0$}
At time step $t=0$, the state encoder $\e$ observes $X_{0}$ and it transmits $S_{e,0}=\sqrt{\frac{\ps}{\alpha_{0}}}X_{0}$. The decoder $\d$ receives $R_0=hS_{e,0}+Z_{d,0}$. It estimates $X_{0}$ as $ \hat{X}_{0}=\frac{1}{h}\sqrt{\frac{\alpha_{0}}{\ps}}R_0=X_{0}+\frac{1}{h}\sqrt{\frac{\alpha_{0}}{\ps}}Z_{d,0}$. The controller $\cc$ then takes an action $U_{0}=-\lm\hat{X}_{0}$ which results in
\begin{align}
X_{1}=\lm X_0 + U_0 + W_0=\lm\left(X_{0}-\hat{X}_{0}\right)+W_0=-\frac{\lm}{h}\sqrt{\frac{\alpha_{0}}{\ps}}Z_{d,0}+W_0.
\end{align}
The new plant state $X_{1}\sim\mathcal{N}(0,\a_{1})$ where $\a_{1}=\frac{\lm^2 \nd}{h^2\ps}\a_{0}+n_w$.
\subsubsection*{First transmission phase, $t=1,3,5,...$}
The state encoder $\e$ observes $X_t$ and transmits $S_{e,t}=\sqrt{\frac{2\b\ps}{\a_t}}X_t$. The relay nodes $\left\{\r_i\right\}^L_{i=1}$ receive this signal over the Gaussian links and do not transmit any signal in this transmission phase due to half-duplex restriction. The decoder $\d$ observes $R_t=hS_{e,t}+Z_{d,t}$ and computes the MMSE estimate of $X_t$, which is given by
\begin{align*}
  \hat{X}_{t}=\ex[X_t|R_{[1,t]}] \stackrel{(a)}{=}\ex[X_t|R_t]\stackrel{(b)}{=}\frac{\ex[X_tR_t]}{\ex[R^2_t]} R_t\stackrel{(c)}{=}\left(\frac{h\sqrt{2\b\ps\a_t}}{2h^2\b\ps+\nd}\right)R_t, \nonumber
\end{align*}
where ($a$) follows from the \emph{orthogonality principle} of MMSE estimation (that is $\mathbb{E}[X_t R_{t-j}]=0$ for $j\geq1$) \cite{HayesBook}; ($b$) follows from the fact that the optimum MMSE estimator for a Gaussian variable is linear \cite{HayesBook}; and ($c$) follows from $\ex[X_tR_t]=\sqrt{2h^2\b\ps\a_t}$ and $\ex[R^2_t]=2h^2\b\ps+\nd$.\newline
The controller $\cc$ takes an action $U_{t}=-\lm\hat{X}_{t}$ which results in $X_{t+1}=\lm(X_{t}-\hat{X}_{t})+W_t$. The new plant state $X_{t+1}$ is a linear combination of zero mean Gaussian variables $\{X_t,\hat{X}_t,W_t\}$, therefore it is also zero mean Gaussian with the following variance
\begin{align}
\a_{t+1}:=\ex[X^2_{t+1}]&=\lm^2\ex[(X_{t}-\hat{X}_{t})^2]+\ex[W^2_t]=\lm^2\left(\frac{\nd}{2h^2\b\ps+\nd}\right)\a_t+n_w, \label{eq:noisyP_varFirstPhase}
\end{align}
where the last equality follows from $\ex[X_t\hat{X}_t]=\ex[\hat{X}^2_t]=\frac{2h^2\b\ps\a_t}{2h^2\b\ps+\nd}$ (by computation).
\subsubsection*{Second transmission phase, $t=2,4,6,...$}
The encoder $\e$ observes $X_t$ and transmits $S_{e,t}=\sqrt{\frac{2\left(1-\b\right)\ps}{\a_t}}X_t$. In this phase the relay nodes choose to transmit their own signal to the decoder $\d$ and thus they cannot listen to the signal transmitted from the state encoder due to the half-duplex assumption. Each relay node amplifies the signal that it had received in the previous time step (first transmission phase) under an average transmit power constraint and transmits it to the decoder $\d$. The signal transmitted from the $i$-th relay node is thus given by
$S^i_{r,t}=\sqrt{\frac{2\pr^i}{\left(2\b\ps+\nr^i \right)}}\left(S_{e,t-1}+Z^i_{r,t-1}\right)$.
The decoder $\mathcal{D}$ accordingly receives
\begin{align}
 R_t&=hS_{e,t}+\sum^L_{i=1}h_iS^i_{r,t}+Z_{d,t}=L_1X_t+L_2X_{t-1}+\tilde{Z}_t,
\end{align}
where $L_1=\sqrt{\frac{2(1-\b)h^2\ps}{\a_{t}}}$, $L_2=\sum^L_{i=1}\sqrt{\frac{4\b h^2_i\ps\pr^i}{(2\b\ps+\nr^i)\a_{t-1}}}$, and  $\tilde{Z}_t=Z_{d,t}+\sum^L_{i=1}\sqrt{\frac{2h^2_i\pr^i}{2\b\ps+\nr^i}}Z^i_{r,t-1}$ is a white Gaussian noise sequence with zero mean and variance $\nt(\b,\{\pr^i\}^L_{i=1})=\nd+\sum^L_{i=1}\frac{2h^2_i\pr^i\nr^i}{2\b\ps+\nr^i}$. The decoder then computes the MMSE estimate of $X_t$ given all previous channel outputs $\{R_1,R_2,...,R_t\}$ in the following three steps:
\begin{enumerate}
  \item Compute the MMSE prediction of $R_t$ from $\{R_1,R_2,...,R_{t-1}\}$, which is given by $\hat{R}_t=L_2\hat{X}_{t-1}$,
  where $\hat{X}_{t-1}$ is the MMSE estimate of $X_{t-1}$.
  \item Compute the innovation
\begin{align}\label{eq:innovation}
  I_t&=R_t-\hat{R}_t=L_1X_t+L_2(X_{t-1}\!-\!\hat{X}_{t-1})\!+\!\tilde{Z}_t\stackrel{(a)}{=}\left(\frac{\lm L_1+L_2}{\lm}\right)X_t\!-\!\frac{L_2}{\lm}W_{t-1}\!+\!\tilde{Z}_t,
 \end{align}
where $(a)$ follows from $X_t=\lm\left(X_{t-1}-\hat{X}_{t-1}\right)+W_{t-1}$.
  \item Compute the MMSE estimate of $X_{t}$ given $\{R_1,R_2,...,R_{t-1},I_t\}$.
The state $X_t$ is independent of $\{R_1,R_2,...,R_{t-1}\}$ given $I_t$, therefore we can compute the estimate $\hat{X}_{t}$ based only on $I_t$ without any loss of optimality, that is,
\begin{align}
\hat{X}_{t}&=\mathbb{E}[X_{t}|I_t]\stackrel{(a)}{=}\frac{\mathbb{E}[X_tI_t]}{\mathbb{E}[I^2_t]}I_t\stackrel{(b)}{=}\frac{\lm\left(\lm L_1+L_2\right)\a_t}{\left(\lm L_1+L_2\right)^2\a_t+L^2_2n_w+\lm^2 \nt(\b,\pr)} I_t,
\end{align}
where ($a$) follows from an MMSE estimation of a Gaussian variable; and ($b$) follows from $\ex[X_tI_t]=\left(\frac{\lm L_1+L_2}{\lm}\right)\a_t$ and $\ex[I^2_t]=\left(\frac{\lm L_1+ L_2}{\lm}\right)^2\a_t+\frac{L^2_2 n_w}{\lm^2}+\nt(\b,\pr)$.
\end{enumerate}
The controller $\cc$ takes action $U_{t}=-\lm\hat{X}_{t}$ which results in $X_{t+1}=\lm(X_{t}-\hat{X}_{t})+W_t$. The new plant state $X_{t+1}$ is a linear combination of zero mean Gaussian random variables $\{X_t,\hat{X}_t,W_t\}$, therefore it is also zero mean Gaussian distributed with the following variance,
\begin{align}
&\a_{t+1}\!=\!\lm^2\ex[(X_{t}\!-\!\hat{X}_{t})^2]\!+\!\ex[W^2_t]\stackrel{(a)}{=}\lm^2\a_t\left(\frac{L^2_2n_w\!+\!\lm^2\nt(\b,\pr)}{\left(\lm L_1\!+\!L_2\right)^2\a_t\!+\!L^2_2n_w\!+\!\lm^2 \nt(\b,\pr) }\right)\!+\!n_w \label{eq:varBet}
\end{align}
\begin{align}
&\stackrel{(b)}{=}\lm^2\a_t\left(\frac{\left(\sum^L_{i=1}\sqrt{\frac{4h^2_i\b\ps\pr^i}{2\b\ps \!+\! \nr^i}}\right)^2\frac{n_w}{\a_{t-1}} \!+\! \lm^2\nt(\b,\pr)}{\left(\lm \sqrt{2h^2(1 \!-\! \b\ps)}\!+\!\sum^L_{i=1}\sqrt{\frac{4h^2_i\b\ps\pr^i}{2\b\ps\!+\!\nr^i}\frac{\a_t}{\a_{t-1}}}\right)^2\!+\!\left(\sum^L_{i=1}\sqrt{\frac{4h^2_i\b\ps\pr^i}{2\b\ps\!+\!\nr^i}}\right)^2\frac{n_w}{\a_{t-1}}\!+\!\lm^2 \nt(\b,\pr) }\right) \nonumber \\
&\!+\! n_w \stackrel{(c)}{=}\lm^2\left(\lm^2 k \a_{t-1}+n_w\right)\left(\frac{\left(n_w k_1\right)\frac{1}{\a_{t-1}}+\lm^2\nt(\b,\pr)}{\left(\lm k_2+\sqrt{\frac{k_1}{\lm^2}(\lm^2k+n_w\frac{1}{\a_{t-1}})}\right)^2+\left(n_w k_1\right)\frac{1}{\a_{t-1}}+\lm^2\nt(\b,\pr)}\right)+n_w \nonumber \\
&=\lm^2\left(\lm^2 k \a_{t-1}+n_w\right)\left(\frac{\left(\frac{n_w k_1}{\lm^2}\right)\frac{1}{\a_{t-1}}+\nt(\b,\pr)}{\left(k_2+\sqrt{k_1 k+\frac{n_w k_1}{\lm^2}\frac{1}{\a_{t-1}}}\right)^2+\left(\frac{n_w k_1}{\lm^2}\right)\frac{1}{\a_{t-1}}+\nt(\b,\pr)}\right)+n_w, \label{eq:noisyP_varSecondPhase}
\end{align}
where $(a)$ follows from $\ex[X_t\hat{X}_t]=\ex[\hat{X}^2_t]=\frac{\left(\lm L_1+L_2\right)^2\a_t}{\left(\lm L_1+L_2\right)^2\a_t+L^2_2n_w+\lm^2 \nt(\b,\pr)}$; $(b)$ follows by substituting the values of $L_1$ and $L_2$; and $(c)$ by substituting $\frac{\a_t}{\a_{t-1}}$ using (\ref{eq:noisyP_varFirstPhase}) and by defining $k:=\frac{\n}{2h^2\b\ps+\n}$, $k_1:=\left(\sum^L_{i=1}\sqrt{\frac{4h^2_i\b\ps\pr^i}{2\b\ps+\nr^i}}\right)^2$, $k_2:=\sqrt{2h^2(1-\b\ps)}$.

We want to find the values of the parameter $\lm$ for which the second moment of the state remains bounded. Rewriting (\ref{eq:noisyP_varFirstPhase}) and (\ref{eq:noisyP_varSecondPhase}), the variance of the state at any time $t$ is given by
\begin{align}
\a_t&=\lm^2\left(\lm^2 k \a_{t-2}+n_w\right)\underbrace{\left(\frac{\left(\frac{n_w k_1}{\lm^2}\right)\frac{1}{\a_{t-2}}+\nt(\b,\pr)}{\left(k_2+\sqrt{k_1 k+\frac{n_w k_1}{\lm^2}\frac{1}{\a_{t-2}}}\right)^2+\left(\frac{n_w k_1}{\lm^2}\right)\frac{1}{\a_{t-2}}+\nt(\b,\pr)}\right)}_{\triangleq f(\a_{t-2})}+n_w \nonumber \\
&=\lm^2\left(\lm^2 k \a_{t-2}+n_w\right) f(\a_{t-2})+ n_w, \qquad t=3,5,7,... \label{eq:varSecondPhase_Gen} \\
\a_t&=\lm^2\left(\frac{N}{2h^2\b\ps+N}\right)\a_{t-1}+n_w , \qquad t=2,4,6,... \label{eq:varFirstPhase_Gen}
\end{align}
%
%
where $\a_{1}=\frac{\lm^2 N}{h^2\ps}\a_{0}+n_w$. If the odd indexed sub-sequence $\{\a_{2t+1}\}$ in (\ref{eq:varSecondPhase_Gen}) is bounded, then
the even indexed sub-sequence $\{\a_{2t}\}$ in (\ref{eq:varFirstPhase_Gen}) is also bounded. Thus it is sufficient to consider the odd indexed sub-sequence $\{\a_{2t+1}\}$. We will now construct a sequence $\{\a'_{t}\}$ which upper bounds the sub-sequence $\{\a_{2t+1}\}$. Then we will derive conditions on the system parameter $\lm$ for which the sequence $\{\a'_{t}\}$ stays bounded and consequently the boundedness of $\{\a_{2t+1}\}$ will be guaranteed. In order to construct the upper sequence $\{\a'_{t}\}$, we work on the term $f(\a_{t-2})$ in \eqref{eq:varSecondPhase_Gen} and make use of the following lemma.
\begin{lemma}(\cite[Lemma 4.1]{zaidiACC11}) \label{lm:halfDuplex}
Consider a function $f(x)=\frac{a+\frac{b}{x}}{(c+\sqrt{d+\frac{b}{x}})^2+a+\frac{b}{x}}$ defined over the interval $[0,\infty)$, where $0\leq a,b,c,d < \infty$. The function $f(x)$ can be upper bounded as $f(x)\leq f_{\infty}+\frac{m}{x}$ for some $0<m<\infty$, where $f_{\infty}:=\lim_{x\rightarrow\infty}f(x)=\frac{a}{(c+\sqrt{d})^2+a}$.
\end{lemma}

Starting from \eqref{eq:varSecondPhase_Gen} and by using the above lemma, we write the following series of inequalities
\begin{align}\label{eq:uppSeqIneqHalfDup}
&\a_t\!=\!\lm^2\left(\lm^2 k \a_{t-2}\!+\!n_w\right) f(\a_{t-2})\!+\!n_w\stackrel{(a)}{\leq}\lm^2\left(\lm^2 k \a_{t-2}\!+\!n_w\right)\left(f_{\infty}\!+\!\frac{m}{\a_{t-2}}\right)\!+\!n_w \!=\!\lm^4kf_{\infty}\a_{t-2}\!+\! \nonumber \\
&\frac{\lm^2n_w m}{\a_{t-2}}\!+\!\lm^2 n_w f_{\infty}\!+\!\lm^4 m k+n_w\stackrel{(b)}{\leq}\lm^4kf_{\infty}\a_{t-2}\!+\!\lm^2 m\!+\! \lm^2 n_w f_{\infty}\!+\!\lm^4 m k+n_w =:g(\a_{t-2}),
\end{align}
where $(a)$ follows from Lemma \ref{lm:halfDuplex} and $f_{\infty}=\lim_{\a\rightarrow\infty}f(\a)=\left(\frac{\nt(\b,\pr)}{(k_2+\sqrt{k_1 k})^2+\nt(\b, \pr)}\right)$; and $(b)$ follows from the fact that $\a_t\geq n_w$ for all $t$ according to \eqref{eq:varFirstPhase_Gen} and \eqref{eq:varSecondPhase_Gen}. Since $g(\a)$ in \eqref{eq:uppSeqIneqHalfDup} is a linearly increasing function, it can be used to construct the sequence $\{\a'_t\}$, which upper bounds the odd indexed sub-sequence $\{\a_{2t+1}\}$ given in \eqref{eq:varSecondPhase_Gen}. We construct the sequence $\{\a'_t\}$ for all $t\geq 1$ as
\begin{align}\label{eq:upperSeqHalfDuplex}
&\a_{2t+1}\leq \a'_{t+1}:=g(\a'_t)\stackrel{(a)}{=}\lm^4kf_{\infty}\a'_{t}+\lm^2 m+ \lm^2 n_w f_{\infty}+\lm^4 m k+n_w  \nonumber \\
&\stackrel{(b)}{=}\left(\lm^4kf_{\infty}\right)^t\a'_{0}+(\lm^2 m+ \lm^2 n_w f_{\infty}+\lm^4 m k+n_w)\sum^{t-1}_{i=0}\left(\lm^4kf_{\infty}\right)^i,
\end{align}
where ($a$) follows from \eqref{eq:uppSeqIneqHalfDup} and ($b$) follows by recursively apply ($a$).

We observe from (\ref{eq:upperSeqHalfDuplex}) that if $\left(\lm^4kf_{\infty}\right)=\left(\frac{\lm^4 k\nt(\b,\pr)}{(k_2+\sqrt{k_1 k})^2+\nt(\b, \pr)}\right)<1$, then the sequence $\{\a'_t\}$ converges as $t\rightarrow\infty$ and consequently the original sequence $\{\a_t\}$ is guaranteed to stay bounded. Thus the system in (\ref{eq:stateEquation}) can be mean square stabilized if
\begin{align}
&\lm^4<\left(\frac{(k_2+\sqrt{k_1 k})^2+\nt(\b, \{\pr^i\}^L_{i=1})}{k\nt(\b,\{\pr^i\}^L_{i=1})}\right) \label{eq:condHalfDup} \\
&\Rightarrow\log\left(\lm\right)\!<\!\frac{1}{4}\left(\log\left(1+\frac{2h^2\b\ps}{\nd}\right)+\log\left(1+\frac{\tilde{M}(\b,\{\pr^i\}^L_{i=1})}{\nt(\b,\{\pr^i\}^L_{i=1})}\right)\right), \label{eq:suffCondIntHalfDup}
\end{align}
where the last equality follows from $k=\frac{\n}{2h^2\b\ps+\n}$ and $M(\b,\{\pr^i\}^L_{i=1}):=(k_2+\sqrt{k_1 k})^2$. Since the relay nodes amplify the desired signal as well as the noise, which is then superimposed at the decoder to the signal coming directly from the state encoder, the optimal choice of the transmit powers $\{\pr^i\}^L_{i=1}:\sum^L_{i=1}\pr^i\leq\pra$ depends on the parameters $\{\ps,\{\nr^i\}^L_{i=1},\nd,h,h_i,\b\}$. Moreover, the optimal choice of the power allocation factor $\b$ at the state encoder also depends on these parameters. Therefore, we rewrite (\ref{eq:suffCondIntHalfDup}) as \eqref{eq:thmHalfDup},
which completes the proof.\hfill $\square$

\subsection{Remark \ref{rem:infoRate} on Information Rate}\label{appendixAchievability}
The given scheme can be seen as a point-point communication channel, where $R_{2t\sm1}$ is the channel output corresponding to the input $S_{e,2t\sm1}$ and $I_{2t}$ is the channel output corresponding to the input $S_{e,2t}$ for $t=1,2,3,\dots$. The information rate
 is given by
\begin{align} \label{eq:infoRate}
&\lim_{T\rightarrow\infty}\frac{1}{2T}I\left(\left\{S_{e,2t\sm1}\right\}^{T}_{t=1},\left\{S_{e,2t}\right\}^{T}_{t=1}; \left\{R_{2t\sm1}\right\}^{T}_{t=1},\left\{I_{2t}\right\}^{T}_{t=1}\right)\nonumber \\
&=\lim_{T\rightarrow\infty}\frac{1}{2T}\bigg[h\left(\left\{R_{2t\sm1}\right\}^{T}_{t=1},\left\{I_{2t}\right\}^{T}_{t=1}\right)-h\left(\left\{R_{2t\sm1}\right\}^{T}_{t=1},\left\{I_{2t}\right\}^{T}_{t=1}|\left\{S_{e,2t\sm1}\right\}^{T}_{t=1},\left\{S_{e,2t}\right\}^{T}_{t=1}\right)\bigg] \nonumber \\
&\stackrel{(a)}{=}\lim_{T\rightarrow\infty}\frac{1}{2T}\bigg[\sum^{T}_{t=1}\left(h\left(R_{2t\sm1}\right)+h\left(I_{2t}\right)-h\left(R_{2t\sm1}|S_{e,2t\sm1}\right)-h\left(I_{2t}|S_{e,2t}\right)\right) \bigg]\nonumber \\
&\stackrel{(b)}{=}\lim_{T\rightarrow\infty}\frac{1}{2T}\bigg[T\big( h\left(R_{2t\sm1}\right)+ h\left(I_{2t}\right)-h\left(R_{2t\sm1}|S_{e,2t\sm1}\right)-h\left(I_{2t}|S_{e,2t}\right)\big)\bigg] \nonumber \\
&=\frac{1}{2}\left(I\left(S_{e,2t\sm1};R_{2t\sm1}\right)+I\left(S_{e,2t};I_{2t}\right)\right),
\end{align}
where $(a)$ follows from the fact that $P(I_{2t},R_{2t\sm1}|S_{e,2t},S_{e,2t\sm1})=P(I_{2t}|S_{e,2t})P(R_{2t\sm1}|S_{e,2t\sm1})$, the channel is memoryless, the random variables are Gaussian and $\ex[R_{2l\sm1}R_{2k\sm1}]=\ex[I_{2l}I_{2k}]=0$ for $k\neq l$, and $\ex[R_{2l\sm1}I_{2k}]=0$ for all $l,k=1,2,3,..$; and
$(b)$ follows from the fact that $R_{2t\sm1}$ and $I_{2t}$ are both sequences of i.i.d. variables. For the first transmission phase the mutual information between the transmitted variable and the received variable is given by
\begin{align} \label{eq:mutualFirst}
I\left(S_{e,2t\sm1};R_{2t\sm1}\right)\!=\!h(R_{2t\sm1})\!-\!h(R_{2t\sm1}|S_{e,2t\sm1})\!=\!h(R_{2t\sm1})\!-\!h(Z_{2t\sm1})\!\stackrel{(a)}{=}\!\frac{1}{2}\log{\left(1\!+\!\frac{2h^2\b\ps}{\n}\right)},
\end{align}
where $(a)$ follows from $R_{2t\sm1}\sim\mathcal{N}(0,2h^2\b\ps+\n)$ and $Z_{2t\sm1}\sim\mathcal{N}(0,\n)$.
In the second phase the decoder computes the innovation $I_t$ according to (\ref{eq:innovation}). The mutual information between the transmitted variable and the innovation variable is then given by
\begin{align}\label{eq:mutualSecond}
I\left(S_{e,2t};I_{2t}\right)&=h(I_{2t})-h(I_{2t}|S_{e,2t})=h(I_{2t})-h(\tilde{Z}_{2t})\stackrel{(a)}{=}\frac{1}{2}\log{\left(1+\frac{\tilde{M}(\b,\pr)}{\nt(\b,\pr)}\right)},
\end{align}
where $(a)$ follows from $I_{2t}\sim\mathcal{N}(0,\tilde{M}(\b,\pr)+\nt(\b,\pr))$ and $\tilde{Z}_{2t}\sim\mathcal{N}(0,\nt(\b,\pr))$. From \eqref{eq:mutualFirst}, \eqref{eq:mutualSecond}, and \eqref{eq:infoRate} the corresponding information rate is equal to
\begin{align}
\frac{1}{4}\left(\log{\left(1+\frac{2h^2\b\ps}{\n}\right)}+\log{\left(1+\frac{\tilde{M}(\b,\pr)}{\nt(\b,\pr)}\right)}\right).
\end{align}
For the given channel, the directed information rate is equal to information rate due to mutual independence of the channel output sequence \cite[Theorem 2]{Massey90}.

\subsection{Proof of Lemma \ref{lm:vectorConverge}}\label{apx:MultiDim}
Assume that $T(x)$ is a non-decreasing mapping with a unique fixed point $x^\star$. Further assume that there exist $u\leq x^\star \leq v$ such that $T(u)\geq u$ and $T(v) \leq v$. Consider a sequence generated by the following iterations: $x_{t+1}=T(x_t)$ with $t\in\mathbb{N}$, $x_0\in\mathbb{R}$. We want to show that starting from any $x_0\in\mathbb{R}$, the sequence $\{x_t\}$ converges. There are three possibilities: i) $x_0=x^\star$, ii) $x_0>x^\star$, and iii) $x_0<x^\star$. For $x_0\in[x^\star,\infty)$ we have $T(x)\leq x$, therefore $x_1=T(x_0)\leq x_0$. Since $T(x)$ is non-decreasing, $x_2=T(x_1)\leq T(x_0)= x_1$. Thus for any $t\in\mathbb{N}$ we have $x_{t+1}=T(x_t)\leq T(x_{t-1})=x_t$. Further this sequence is lower bounded by $x^\star$ because for any $x_t\in [x^\star,\infty)$, $x^\star=T(x^\star)\leq T(x_t)=x_{t+1}$ due to non-decreasing $T(\cdot)$. Thus the sequence $\{x_t\}$ converges since it is monotonically decreasing and lower bounded by $x^\star$\cite[Theorem 3.14]{RudinBook}. For $x\in(-\infty,x^\star]$ we have $T(x)\geq x$, therefore $x_1=T(x_0)\leq x_0$. Since $T(x)$ is non-decreasing, we have $x_2=T(x_1)\geq T(x_0)= x_1$. Thus for any $t\in\mathbb{N}$ we have $x_{t+1}=T(x_t)\geq T(x_{t-1})=x_t$. Further this sequence is upper bounded by $x^\star$ because for any $x_t\in (-\infty,x^\star]$, we have $x_{t+1}=T(x_t)\leq T(x^\star) = x^\star$ due to non-decreasing $T(\cdot)$.  Since $\{x_t\}$ is strictly increasing and upper bounded by $x^\star$ for $x_0\in[x^\star,\infty)$, it converges \cite[Theorem 3.14]{RudinBook}. \hfill $\square$

\bibliographystyle{IEEEtran}
\bibliography{Feb8}

\end{document}